\newif\ifentcs
\entcsfalse
\ifentcs
\documentclass[9pt]{entcs}
\usepackage{entcsmacro}
\else
\documentclass{article}
\usepackage{amssymb}

\newenvironment{ack}{\textbf{Acknowledgments.}}{\vskip0.5em}

\newtheorem{theorem}{Theorem}[section]
\newtheorem{proposition}[theorem]{Proposition}
\newtheorem{corollary}[theorem]{Corollary}
\newtheorem{lemma}[theorem]{Lemma}
\newtheorem{fact}[theorem]{Fact}
\newenvironment{proof}{\emph{Proof.}}{\relax\hfill\qed}
\newtheorem{definition}[theorem]{Definition}
\newtheorem{remark}[theorem]{Remark}
\newcommand\qed{\hfill$\Box$\vskip0.2em}
\fi

\usepackage{url}
\usepackage{latexsym}
\usepackage{stmaryrd,amsmath,amsfonts}
\usepackage{graphicx}
\usepackage[all]{xypic}
\newcommand\eqdef{\mathrel{\buildrel \text{def}\over=}}
\newcommand{\pow}{\mathbb{P}}

\newcommand{\real}{\mathbb{R}}
\newcommand\Rplus{\real_+}
\newcommand\creal{\overline\real_+}
\newcommand\creals{\overline\real_{+\sigma}}
\newcommand\Lform{\mathcal L}
\newcommand\cb[1]{\mathbf{#1}} % or \pmb
\newcommand{\Topcat}{\cb{Top}}
\newcommand{\Loc}{\cb{Loc}}
\newcommand{\Frm}{\cb{Frm}}
\newcommand\directed{\sideset{}{^{\,\makebox[0pt]{$\scriptstyle\uparrow$}\!}}}
\newcommand\filtered{\sideset{}{^{\,\makebox[0pt]{$\scriptstyle\downarrow$}\!}}}
\newcommand\dsup{\directed\sup}

\newcommand\dcup{\directed\bigcup}
\newcommand\fcap{\filtered\bigcap}

\newcommand\diff{\setminus}
\newcommand\nat{\mathbb{N}}
\newcommand\Z{\mathbb{Z}}
\newcommand\upc{\mathop{\uparrow}}
\newcommand\uuarrow{\rlap{$\uparrow$}\raise.5ex\hbox{$\uparrow$}}%%
\newcommand\dc{\mathop{\downarrow}}
\newcommand\ddarrow{\rlap{$\downarrow$}\raise.5ex\hbox{$\downarrow$}}%%
\newcommand\interior[1]{int ({#1})}
\newcommand\limp{\mathrel{\Rightarrow}}
\newcommand\liff{\mathrel{\Leftrightarrow}}
\newcommand\img{\mathop{\text{Im}}}
\newcommand\Prev{\mathbb{P}}
\newcommand\Angel{{\mathtt{A}}}

\newcommand\Nature{{\mathtt{P}}}
\newcommand\AN{{\Angel\Nature}}

\newcommand\Hoare{{\mathcal H}}

\newcommand\Val{{\mathbf V}}
\newcommand\wk{\text{w}}
\newcommand\Borel[1]{{\mathcal B} (#1)}
\newcommand\rat{\mathbb{Q}}
\newcommand\Open{\mathop{\mathcal O}}
\newcommand\pt{\mathop{\mathsf{pt}}}
\newcommand\bPi{\boldsymbol\Pi}
\newcommand\Sierp{\mathbb{S}}

\newenvironment{recapthm}[1]{\noindent \textbf{Theorem~\ref{#1} (recap).}}

 \newcommand\ForAuthors[1]%          %  temporary remark for the
 {\par\smallskip                     %  authors:
  \begin{center}%                    %
   \fbox%                            %    --------
   {\parbox{0.9\linewidth}%          %    |  #1  |
    {\raggedright\sc--- #1}%         %    --------
   }%                                %
  \end{center}%                      %
  \par\smallskip                     %
 }

 \newcommand\papertitle{Domain-complete and LCS-complete spaces}
 \newcommand\paperabs{%
    We study $G_\delta$ subspaces of continuous dcpos, which we call
    domain-complete spaces, and $G_\delta$ subspaces of locally
    compact sober spaces, which we call LCS-complete spaces.  Those
    include all locally compact sober spaces---in particular, all
    continuous dcpos---, all topologically complete spaces in the
    sense of \v{C}ech, and all quasi-Polish spaces---in particular,
    all Polish spaces.  We show that LCS-complete spaces are sober,
    Wilker, compactly Choquet-complete, completely Baire, and
    $\odot$-consonant---in particular, consonant; that the
    countably-based LCS-complete (resp., domain-complete) spaces are
    the quasi-Polish spaces exactly; and that the metrizable
    LCS-complete (resp., domain-complete) spaces are the completely
    metrizable spaces.  We include two applications: on LCS-complete
    spaces, all continuous valuations extend to measures, and
    sublinear previsions form a space homeomorphic to the convex Hoare
    powerdomain of the space of continuous valuations.
  }
  \newcommand\kuone{%
    Graduate School of Human and Environmental Studies
  }
  \newcommand\kutwo{%
    Kyoto University, Kyoto, Japan
  }
  \newcommand\jsps{%
    The first author was supported by JSPS Core-to-Core Program,
    A. Advanced Research Networks and by JSPS KAKENHI Grant Number 18K11166.
  }
  \newcommand\digicosme{%
    This research was partially supported by Labex
    DigiCosme (project ANR-11-LABEX-0045-DIGICOSME) operated by ANR as
    part of the program ``Investissement d'Avenir'' Idex Paris-Saclay
    (ANR-11-IDEX-0003-02).
  }
  \newcommand\lsv{%
    LSV, ENS Paris-Saclay, CNRS, Universit\'e Paris-Saclay,
    France
    }
\ifentcs
\begin{document}
\begin{frontmatter}
  \title{\papertitle}
  \author[KU]{Matthew de Brecht
    \thanksref{JSPS}\thanksref{dbemail}}
  \address[KU]{\kuone, \kutwo}
  \author[ENSC]{Jean Goubault-Larrecq %\thanksref{ALL}
    \thanksref{Digicosme}\thanksref{myemail}}
  \author[ENSC]{Xiaodong Jia
    \thanksref{Digicosme}\thanksref{xjemail}}
  \author[ENSC]{Zhenchao Lyu
    \thanksref{Digicosme}\thanksref{zlemail}}
  \address[ENSC]{\lsv}
  % \thanks[ALL]{Thanks to everyone who should be thanked}
  \thanks[JSPS]{\jsps}
  \thanks[dbemail]{Email:
    \href{mailto: matthew@i.h.kyoto-u.ac.jp} {\texttt{\normalshape
        matthew@i.h.kyoto-u.ac.jp}}}
  \thanks[Digicosme]{\digicosme}
  \thanks[myemail]{Email:
    \href{mailto: goubault@ens-paris-saclay.fr} {\texttt{\normalshape
        goubault@ens-paris-saclay.fr}}}
  \thanks[xjemail]{Email:
    \href{mailto: jia@lsv.fr} {\texttt{\normalshape jia@lsv.fr}}}
  \thanks[zlemail]{Email:
    \href{mailto:zhenchaolyu@gmail.com} {\texttt{\normalshape
        zhenchaolyu@gmail.com}}}
  \begin{abstract}
    \paperabs
  \end{abstract}
\end{frontmatter}
\else
\title{\papertitle}
\author{Matthew de Brecht\thanks{\jsps}\\
  \url{matthew@i.h.kyoto-u.ac.jp}
  \\ \kuone \\ \kutwo \\
  \\
  Jean Goubault-Larrecq \\
  \url{goubault@ens-paris-saclay.fr} \\[0.5ex]
  Xiaodong Jia$^\dagger$ \\
  \url{jia@lsv.fr} \\[0.5ex]
  Zhenchao Lyu\thanks{\digicosme} \\
  \url{zhenchaolyu@gmail.com} \\[0.5ex]
  \lsv
  }
\begin{document}
\maketitle

\begin{abstract}
  \paperabs
\end{abstract}
\fi

\section{Motivation}
\label{sec:intro}

Let us start with the following question: for which class of
topological spaces $X$ is it true that every (locally finite)
continuous valuation on $X$ extends to a measure on $X$, with its
Borel $\sigma$-algebra?  The question is well-studied, and Klaus
Keimel and Jimmie Lawson have rounded it up nicely in
\cite{KL:measureext}.  A result by Mauricio Alvarez-Manilla \emph{et
  al.\/} \cite{alvarez-manilla00} (see also Theorem~5.3 of the paper
by Keimel and Lawson) states that every locally compact sober space
fits.

Locally compact sober spaces are a pretty large class of spaces,
including many non-Hausdorff spaces, and in particular all the
continuous dcpos of domain theory.  However, such a result will be of
limited use to the ordinary measure theorist, who is used to working
with Polish spaces, including such spaces as Baire space $\nat^\nat$,
which is definitely not a locally compact space.

It is not too hard to extend the above theorem to the following larger
class of spaces (and to drop the local finiteness assumption as
well):%
\newcommand\propextaux{%
  Every continuous valuation $\nu$ on $X$ extends to a measure on $X$ with
  its Borel $\sigma$-algebra.}
\newcommand\propext{%
  Let $X$ be a (homeomorph of a) $G_\delta$ subset of a locally
  compact sober space $Y$.  \propextaux }
\begin{theorem}
  \label{thm:ext}
  \propext
\end{theorem}
We defer the proof of that result to
Section~\ref{sec:extens-cont-valu}.  The point is that we do have a
measure extension theorem on a class of spaces that contains both the
continuous dcpos of domain theory and the Polish spaces of topological
measure theory.  We will call such spaces \emph{LCS-complete}, and we
are aware that this is probably not an optimal name.  \emph{Topologically
  complete} would have been a better name, if it had not been taken
already \cite{Cech:bicomplete}.

Another remarkable class of spaces is the class of \emph{quasi-Polish}
spaces, discovered and studied by the first author
\cite{deBrecht:qPolish}.  This one generalizes both
$\omega$-continuous dcpos and Polish spaces, and we will see in
Section~\ref{sec:quasi-polish-spaces} that the class of LCS-complete
spaces is a proper superclass.  We will also see that there is no
countably-based LCS-complete space that would fail to be
quasi-Polish.  Hence LCS-complete spaces can be seen as an extension
of the notion of quasi-Polish spaces, and the extension is strict only
for non-countably based spaces.

\begin{figure}
  \centering
  \includegraphics[scale=0.4]{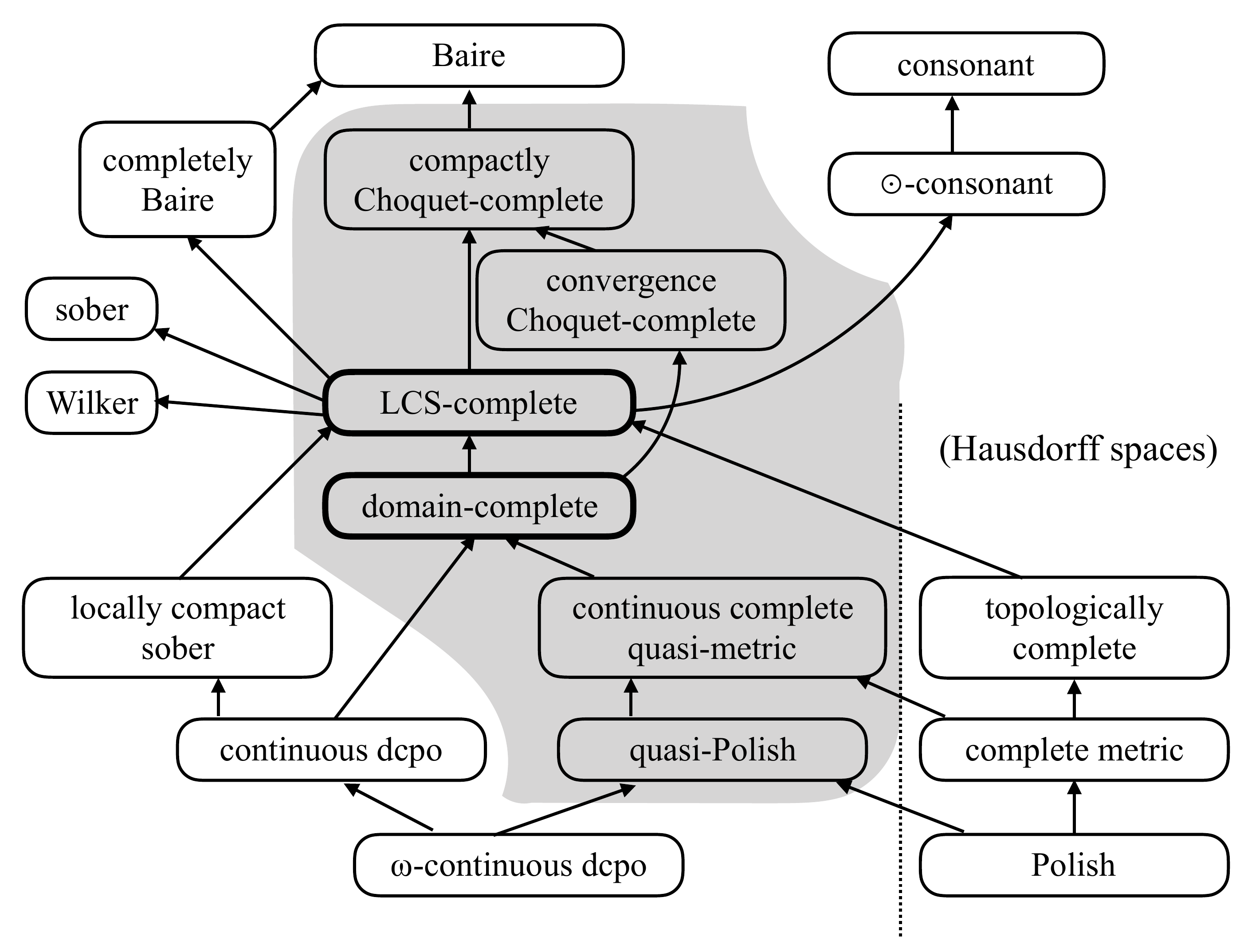}
  \caption{Domain-complete and LCS-complete spaces in relation to
    other classes of spaces}
  \label{fig:panorama}
\end{figure}

Generally, our purpose is to locate LCS-complete spaces, as well as
the related \emph{domain-complete} spaces inside the landscape formed
by other classes of spaces.  The result is summarized in
Figure~\ref{fig:panorama}.  The gray area is indicative of what
happens with countably-based spaces: for such spaces, all the classes
inside the the gray area coincide.

We proceed as follows.  Once we have recalled some background in
Section~\ref{sec:preliminaries}, and given the basic definitions in
Section~\ref{sec:defin-basic-prop}, we show that continuous complete
quasi-metric spaces, quasi-Polish spaces and topologically complete
spaces are all LCS-complete in
Sections~\ref{sec:compl-quasi-metr}--\ref{sec:vcech-compl-spac}.  Then
we show that all LCS-complete spaces are sober
(Section~\ref{sec:sobriety}), Wilker
(Section~\ref{sec:wilker-condition}), Choquet-complete and in fact a
bit more (Section~\ref{sec:choq-compl-baire}), Baire and even
completely Baire (Section~\ref{sec:baire-property}), consonant and
even $\odot$-consonant (Section~\ref{sec:consonance}).  In the
process, we explore the Stone duals of domain-complete and
LCS-complete spaces in Section~\ref{sec:stone-duality-domain}.  While
the class of LCS-complete spaces is strictly larger than the class of
domain-complete spaces, in Section~\ref{sec:choq-compl-baire}, we also
show that for countably-based spaces, LCS-complete, domain-complete,
and quasi-Polish are synonymous.  We give a first application in
Section~\ref{sec:space-lform-x}: when $X$ is LCS-complete, the Scott
and compact-open topologies on the space $\Lform X$ of lower
semicontinuous maps from $X$ to $\creal$ coincide; hence $\Lform X$
with the Scott topology is locally convex, allowing us to apply an
isomorphism theorem \cite[Theorem~4.11]{JGL-mscs16} beyond
core-compact spaces, to the class of all LCS-complete spaces.  In the
sequel (Sections~\ref{sec:limits}--\ref{sec:fail-cart-clos}), we
explore the properties of the categories of LCS-complete, resp.\
domain-complete spaces: countable products and arbitrary coproducts
exist and are computed as in topological spaces, but those categories
have neither equalizers nor coequalizers, and are not
Cartesian-closed; we also characterize the exponentiable objects in
the category of quasi-Polish spaces as the countably-based locally
compact sober spaces.  Section~\ref{sec:compact-subsets-lcs} is of
independent interest, and characterizes the compact saturated subsets
of LCS-complete spaces, in a manner reminiscent of a well-known
theorem of Hausdorff on complete metric spaces.  We prove
Theorem~\ref{thm:ext} in Section~\ref{sec:extens-cont-valu}, and we
conclude in Section~\ref{sec:conclusion}.

\begin{ack}
  % We thank Xiaodong Jia for finding a simplification in the proof of
  % Proposition~\ref{prop:cChoquet:2nd}.
  The second author thanks Szymon Dolecki for pointing him to
  \cite[Proposition~7.3]{DGL:consonant}.
\end{ack}

\section{Preliminaries}
\label{sec:preliminaries}

We assume that the reader is familiar with domain theory
\cite{GHKLMS:contlatt,AJ:domains}, and with basic notions in
non-Hausdorff topology \cite{JGL-topology}.

We write $\Topcat$ for the category of topological spaces and
continuous maps.

$\Rplus$ denotes the set of non-negative real numbers, and $\creal$ is
$\Rplus$ plus an element $\infty$, larger than all others.  We write
$\leq$ for the underlying preordering of any preordered set, and for
the specialization preordering of a topological space.  The notation
$\upc A$ denotes the upward closure of $A$, and $\dc A$ denotes its
downward closure.  When $A=\{y\}$, this is simply written $\upc y$,
resp.\ $\dc y$.  We write $\dcup$ for directed unions, $\dsup$ for
directed suprema, and $\fcap$ for filtered intersections.

Compactness does not imply separation, namely, a compact set is one
such that one can extract a finite subcover from any open cover.  A
\emph{saturated} subset is a subset that is the intersection of its
open neighborhoods, equivalently that is an upwards-closed subset in
the specialization preordering.

We write $\ll$ for the way-below relation on a poset $Y$, and
$\uuarrow y$ for the set of points $z \in Y$ such that $y \ll z$.

We write $\interior A$ for the interior of a subset $A$ of a
topological space $X$, and $\Open X$ for its lattice of open subsets.

A space is \emph{locally compact} if and only if every point has a
base of compact saturated neighborhoods.  It is \emph{sober} if and
only if every irreducible closed subset is the closure of a unique
point.  It is \emph{well-filtered} if and only if given any filtered
family ${(Q_i)}_{i \in I}$ of compact saturated subsets and every open
subset $U$, if $\bigcap_{i \in I} Q_i \subseteq U$ then
$Q_i \subseteq U$ for some $i \in I$.  In a well-filtered space, the
intersection $\bigcap_{i \in I} Q_i$ of any such filtered family is
compact saturated.  Sobriety implies well-filteredness, and the two
properties are equivalent for locally compact spaces.

A space $X$ is \emph{core-compact} if and only if $\Open X$ is a
continuous lattice.  Every locally compact space is core-compact, and
in that case the way-below relation on open subsets is given by
$U \Subset V$ if and only if $U \subseteq Q \subseteq V$ for some
compact saturated set
$Q$.  %, see \cite[Section~5.2.1]{JGL-topology} for example.
Conversely, every core-compact sober space is locally compact.

\section{Definition and basic properties}
\label{sec:defin-basic-prop}

A $G_\delta$ subset of a topological space $Y$ is the intersection of
a countable family ${(W_n)}_{n \in \nat}$ of open subsets of $Y$.
Replacing $W_n$ by $\bigcap_{i=0}^n W_i$ if needed, we may assume that
the family is \emph{descending}, namely that $W_0 \supseteq W_1
\supseteq \cdots \supseteq W_n \cdots$.

\begin{definition}
  \label{defn:ccqm}
  A \emph{domain-complete space} is a (homeomorph of a) $G_\delta$ subset of a
  continuous dcpo, with the subspace topology from the Scott topology.

  An \emph{LCS-complete space} is a (homeomorph of a) $G_\delta$
  subset of a locally compact sober space, with the subspace topology.
\end{definition}

\begin{remark}
  \label{rem:ccqm:pattern}
  There is a pattern here.  For a class $\mathcal C$ of topological
  spaces, one might call \emph{$\mathcal C$-complete} any homeomorph
  of a $G_\delta$ subset of a space in $\mathcal C$.  For example, if
  $\mathcal C$ is the class of stably (locally) compact spaces, we
  would obtain \emph{SC-complete} (resp., \emph{SLC-complete}) spaces.
  By an easy trick which we shall use in
  Lemma~\ref{lemma:ccqm:compact}, SC-complete and SLC-complete are the
  same notion.
\end{remark}

\begin{proposition}
  \label{prop:ccqm:easy}
  Every locally compact sober space is LCS-complete, in particular
  every quasi-continuous dcpo is LCS-complete.  Every continuous dcpo
  is domain-complete.  Every domain-complete space is LCS-complete.
\end{proposition}
\begin{proof}
  Every space is $G_\delta$ in itself.  Every quasi-continuous dcpo is
  locally compact (being locally finitary compact
  \cite[Exercise~5.2.31]{JGL-topology}) and sober
  \cite[Exercise~8.2.15]{JGL-topology}.  The last part follows from
  the fact that every continuous dcpo is locally compact and sober---for
  example, because it is quasi-continuous.
  % the second part, we note that the continuous dcpos are exactly the
  % sober c-spaces in their Scott topology
  % \cite[Proposition~8.3.36]{JGL-topology}, and that every c-space is
  % trivial locally compact---a c-space is a space where every point has
  % a base of compact neighborhoods of the form $\upc y$, for a single
  % point $y$.  The same argument applies to quasi-continuous dcpos,
  % which are exactly the sober locally finitary compact spaces in their
  % Scott topology \cite[Exercise~8.3.39]{JGL-topology}---a space is
  % locally finitary compact if and only if every point has a base of
  % compact neighborhoods of the form $\upc A$, $A$ finite.
\end{proof}

We will see other examples of domain-complete spaces in
Sections~\ref{sec:compl-quasi-metr}, \ref{sec:quasi-polish-spaces},
and \ref{sec:vcech-compl-spac}.

\begin{remark}
  \label{rem:measurement}
  Given any continuous dcpo (resp., locally compact sober space) $Y$,
  and any descending family ${(W_n)}_{n \in \nat}$ of open subsets of
  $Y$, $X \eqdef \fcap_{n \in \nat} W_n$ is domain-complete (resp.,
  LCS-complete).  We can then define $\mu \colon Y \to \creal$ by
  $\mu (y) \eqdef \inf \{1/2^n \mid y \in W_n\}$.  This is continuous
  from $Y$ to $\creal^{op}$, i.e., $\creal$ with the Scott topology of
  the reverse ordering $\geq$.  Indeed, $\mu^{-1} ([0, a)) = W_n$
  where $n$ is the smallest natural number such that $1/2^n < a$.
  Then $X$ is equal to the \emph{kernel}
  $\ker \mu \eqdef \mu^{-1} (\{0\})$ of $\mu$.  Conversely, any space
  that is (homeomorphic to) the kernel of some continuous map
  $\mu \colon Y \to \creal^{op}$ from a continuous dcpo (resp.,
  locally compact space) $Y$ is equal to
  $\fcap_{n \in \nat} \mu^{-1} ([0, 1/2^n))$, hence is domain-complete
  (resp., LCS-complete).  This should be compared with Keye Martin's
  notion of \emph{measurement} \cite{Martin:measurement}, which is a
  map $\mu$ as above with the additional property that for every
  $x \in \ker \mu$, for every open neighborhood $V$ of $x$ in $Y$,
  there is an $\epsilon > 0$ such that
  $\dc x \cap \mu^{-1} ([0, \epsilon)) \subseteq V$.
\end{remark}
% \begin{proposition}
%   \label{prop:ccqm:1stcount}
%   Every domain-complete space is first-countable.
% \end{proposition}
% \begin{proof}
%   Let $X$ be the intersection of a descending sequence ${(W_n)}_{n \in
%     \nat}$ of open subsets of a continuous dcpo $Y$, and pick $x \in
%   X$.  Since $Y$ is continuous, $x$ is the supremum of a directed family
%   ${(x_i)}_{i \in I}$ of elements way-below $x$.  Since $W_0$ is open
%   and $x \in W_0$,
%   one of them is in $W_0$, say $x_{i_0}$.  Since $W_1$ is open and $x
%   \in W_1 \cap \uuarrow x_{i_0}$, there is an $i_1 \in I$ such that
%   $x_{i_1} \in W_1$ and $x_{i_0} \ll x_{i_1}$.  We continue this way
%   and obtain a sequence $x_{i_0} \ll x_{i_1} \ll \cdots$ of points
%   way-below $x$ in $Y$, such that $x_{i_k} \in W_k$ for every $k \in
%   \nat$.

%   Since $x_{i_k} \ll x$ for every $k$, $\uuarrow x_{i_k} \cap X$ is an
%   open neighborhood of $x$ in $X$.  For every open neighborhood $U$ of
%   $x$ in $X$, let $V$ be some open subset of $Y$ such that $V \cap X = U$.
% \end{proof}

\section{Continuous complete quasi-metric spaces}
\label{sec:compl-quasi-metr}

A \emph{quasi-metric} on a set $X$ is a map
$d \colon X \times X \to \creal$ satisfying the laws: $d (x,x)=0$;
$d(x,y)=d(y,x)=0$ implies $x=y$; and $d (x,z) \leq d (x,y) + d (y,z)$
(\emph{triangular inequality}).  The pair $X, d$ is then called a
\emph{quasi-metric space}.

Given a quasi-metric space, one can form its poset $\mathbf B (X, d)$
of \emph{formal balls}.  Its elements are pairs $(x, r)$ with
$x \in X$ and $r \in \Rplus$, and are ordered by
$(x, r) \leq^{d^+} (y, s)$ if and only if $d (x, y) \leq r-s$.
Instead of spelling out what a complete (a.k.a.,
\emph{Yoneda-complete} quasi-metric space) is, we rely on the
Kostanek-Waszkiewicz Theorem \cite{KW:formal:ball} (see also
\cite[Theorem~7.4.27]{JGL-topology}), which characterizes them in
terms of $\mathbf B (X, d)$: $X, d$ is \emph{complete} if and only if
$\mathbf B (X, d)$ is a dcpo.

We will also say that $X, d$ is a \emph{continuous complete}
quasi-metric space if and only if $\mathbf B (X, d)$ is a continuous
dcpo.  This is again originally a theorem, not a definition
\cite[Theorem~3.7]{GLN-lmcs17}.  The original, more complex
definition, is due to Mateusz Kostanek and Pawe\l{} Waszkiewicz.

There is a map $\eta \colon X \to \mathbf B (X, d)$ defined by
$\eta (x) \eqdef (x, 0)$.  The coarsest topology that makes $\eta$
continuous, once we have equipped $\mathbf B (X, d)$ with its Scott
topology, is called the \emph{$d$-Scott topology} on $X$
\cite[Definition~7.4.43]{JGL-topology}.  This is our default topology
on quasi-metric spaces, and turns $\eta$ into a topological embedding.

The $d$-Scott topology coincides with the usual open ball topology
when $d$ is a metric (i.e., $d (x,y)=d (y,x)$ for all $x, y$) or when
$X, d$ is a so-called Smyth-complete quasi-metric space
\cite[Propositions~7.4.46, 7.4.47]{JGL-topology}.  We will not say
what Smyth-completeness is (see Section~7.2, ibid.), except that every
Smyth-complete quasi-metric space is continuous complete, by the
Romaguera-Valero theorem \cite{RV:formal:ball} (see also
\cite[Theorem~7.3.11]{JGL-topology}).

\begin{theorem}
  \label{thm:ccqm:gdelta}
  For every continuous complete quasi-metric space $X, d$, the space
  $X$ with its $d$-Scott topology is domain-complete.
\end{theorem}
\begin{proof}
  Every standard quasi-metric space $X, d$ embeds as a $G_\delta$ set
  into $\mathbf B (X, d)$ \cite[Proposition~2.6]{GLN-lmcs17}, and
  every complete quasi-metric space is standard (Proposition~2.2,
  ibid.)  Explicitly, $X$ is homeomorphic to
  $\bigcap_{n \in \nat} W_n$ where
  $W_n \eqdef \{(x, r) \in \mathbf B (X, d) \mid r < 1/2^n\}$ is
  Scott-open.  Since $X, d$ is continuous complete, $\mathbf B (X, d)$
  is a continuous
  dcpo. %, hence in particular locally compact and sober.
\end{proof}

When $d$ is a metric, $\mathbf B (X, d)$ is a continuous poset
\cite[Corollary~10]{edalat98}, with $(x, r) \ll (y, s)$ if and only if
$d (x, y) < r-s$; also, $\mathbf B (X, d)$ is a dcpo if and only if
$X, d$ is complete in the usual, Cauchy sense
\cite[Theorem~6]{edalat98}.  Hence every complete metric space is
continuous complete in our sense.
\begin{corollary}
  \label{corl:ccqm:metric}
  Every complete metric space is domain-complete in its open ball
  topology.  \qed
\end{corollary}

\section{Quasi-Polish spaces}
\label{sec:quasi-polish-spaces}

Quasi-Polish spaces were introduced in \cite{deBrecht:qPolish}, and
can be defined in many equivalent ways.  The original definition is: a
quasi-Polish space is a separable Smyth-complete quasi-metric space
$X, d$, seen as a topological space with the open ball topology.  By
\emph{separable} we mean the existence of a countable dense subset in
$X$ with the open ball topology of $d^{sym}$, where $d^{sym}$ is the
symmetrized metric $d^{sym} (x, y) \eqdef \max (d (x, y), d (y, x))$.
By a lemma due to K\"unzi \cite{Kunzi:quasiuniform}, a quasi-metric
space is separable if and only if its open ball topology is
countably-based.

Since Smyth-completeness implies continuous completeness and also that
the open ball and $d$-Scott topologies coincide
\cite[Theorem~7.4.47]{JGL-topology}, Theorem~\ref{thm:ccqm:gdelta}
implies:
\begin{proposition}
  \label{prop:qPolish}
  Every quasi-Polish space is domain-complete.
\end{proposition}
Not every domain-complete space is quasi-Polish.  In fact, the
following remark implies that not every domain-complete space is even
first-countable.  We will see that all countably-based domain-complete
spaces \emph{are} quasi-Polish in Section~\ref{sec:choq-compl-baire}.

\begin{remark}
  \label{rem:powI}
  Let us fix an uncountable set $I$, and let
  $X \eqdef Y \eqdef \pow (I)$, with the Scott topology of inclusion.
  This is an algebraic, hence continuous dcpo, hence a domain-complete
  space.  $I$ is its top element.  We claim that every collection of
  open neighborhoods of $I$ whose intersection is $\{I\}$ must be
  uncountable.  Imagine we had a countable collection
  ${(V_n)}_{n \in \nat}$ of open neighborhoods of $I$ whose
  intersection is $\{I\}$.  For each $n \in \nat$, $I$ is in some
  basic open set
  $\upc A_n \eqdef \{B \in \pow (I) \mid A_n \subseteq B\}$ (where
  each $A_n$ is a finite subset of $I$) included in $V_n$.  Then
  $\bigcap_{n \in \nat} \upc A_n$ is still equal to $\{I\}$.  However,
  $\bigcap_{n \in \nat} \upc A_n = \upc A_\infty$, where $A_\infty$ is
  the countable set $\bigcup_{n \in \nat} A_n$, and must contain some
  (uncountably many) points other than $I$.
\end{remark}
% For example,
% $\pow (I)$ with the Scott topology of inclusion is always a continuous
% (even algebraic) dcpo, hence is domain-complete.  But $\pow (I)$ is
% not first-countable if $I$ is uncountable
% \cite[Exercise~6.5.9]{JGL-topology}, hence not countably-based, hence
% not quasi-Polish in that case.

\section{Topologically complete spaces}
\label{sec:vcech-compl-spac}

In 1937, Eduard \v{C}ech defined \emph{topologically complete} spaces
as those topological spaces that are $G_\delta$ subsets of some
compact Hausdorff space, or equivalently those completely regular
Hausdorff spaces that are $G_\delta$ subsets of their Stone-\v{C}ech
compactification \cite{Cech:bicomplete}, and proved that a metrizable
space is completely metrizable if and only if it is topologically
complete.

The following is then clear:
\begin{fact}
  \label{fact:top:complete}
  Every topologically complete space in \v{C}ech's sense is LCS-complete.
\end{fact}

% Zden\v{e}k Frol\'\i k \cite{Frolik:Cech:complete} then showed that,
% for completely regular $T_1$ (i.e., $T_{3\;1/2}$) spaces, topological
% completeness is equivalent to the following property, called
% \emph{\v{C}ech-completeness} in \cite[Exercise~7.6.21]{JGL-topology}:
% there is a sequence $\mathcal C_*$ of open covers
% $\mathcal C_n \eqdef {(U_{ni})}_{i \in I_n}$, $n \in \nat$, such that
% every filtered family ${(F_j)}_{j \in J}$ of non-empty closed subsets
% controlled by $\mathcal C_*$ has non-empty intersection;
% ${(F_j)}_{j \in J}$ is \emph{controlled} by $\mathcal C_*$ if and only
% if for every $n \in \nat$, there is a $j \in J$ such that $F_j$ is
% included in some open set $U_{ni}$, $i \in I_n$, from the $n$th cover
% $\mathcal C_n$.

% \ForAuthors{JGL: generalize Fact~\ref{fact:top:complete}, for
%   \v{C}ech-complete spaces, maybe just regular, not even $T_1$?  are
%   they domain-complete?}

\section{Sobriety}
\label{sec:sobriety}

A $\bPi^0_2$ subset of a topological space $Y$ is a space of the form
$\{y \in Y \mid \forall n \in \nat, y \in U_n \limp y \in V_n\}$,
where $U_n$ and $V_n$ are open in $Y$.  Every $G_\delta$ subset of $Y$
is $\bPi^0_2$ (take $U_n=Y$ for every $n$), and every closed subset of
$Y$ is $\bPi^0_2$ (take $U_n$ equal to the complement of that closed
subset for every $n$, and $V_n$ empty).  More generally, we consider
\emph{Horn} subsets of $Y$, defined as sets of the form
$\{y \in Y \mid \forall i \in I, y \in U_i \limp y \in V_i\}$, where
$U_i$, $V_i$ are (not necessarily countably many) open subsets of $Y$.
\begin{proposition}
  \label{prop:ccqm:sober}
  Every Horn subset $X$ of a sober space $Y$ is sober.  In particular,
  every LCS-complete space is sober.
\end{proposition}
\begin{proof}
  We prove the first part.  In the case of $\bPi^0_2$ subsets, that
  was already proved in \cite[Lemma~4.2]{deBrecht:Hurewicz}.  Let
  $X \eqdef \{y \in Y \mid \forall i \in I, y \in U_i \limp y \in
  V_i\}$, with $U_i$ and $V_i$ open.  $\pow (I)$, with the inclusion
  ordering, is an algebraic dcpo, whose finite elements are the finite
  subsets of $I$.  Let $f \colon Y \to \pow (I)$ map $y$ to
  $\{i \in I \mid y \in U_i\}$, and $g$ map $y$ to
  $\{i \in I \mid y \in U_i \cap V_i\}$.  Both are continuous, since
  $f^{-1} (\upc \{i_1, \cdots, i_k\}) = \bigcap_{j=1}^k U_{i_j}$ and
  $g^{-1} (\upc \{i_1, \cdots, i_k\}) = \bigcap_{j=1}^k U_{i_j} \cap
  V_{i_j}$ are open.  The equalizer of $f$ and $g$ is
  $\{y \in Y \mid f (y)=g (y)\} = \{y \in Y \mid \forall i \in I, y
  \in U_i \liff y \in U_i \cap V_i\} = X$, with the subspace topology.
  But every equalizer of continuous maps from a sober space to a $T_0$
  topological space is sober \cite[Lemma~8.4.12]{JGL-topology} (note
  that ``$T_0$'' is missing from the statement of that lemma, but
  $T_0$-ness is required).
\end{proof}

\section{The Wilker condition}
\label{sec:wilker-condition}

A space $X$ satisfies \emph{Wilker's condition}, or is \emph{Wilker},
if and only if every compact saturated set $Q$ included in the union
of two open subsets $U_1$ and $U_2$ of $X$ is included in the union of
two compact saturated sets $Q_1 \subseteq U_1$ and
$Q_2 \subseteq U_2$.  The notion is used by Keimel and Lawson
\cite[Theorem~6.5]{KL:measureext}, and is due to Wilker
\cite[Theorem~3]{Wilker:hom}.  Theorem~8 of the latter states that
every KT$_4$ space, namely every space in which every compact subspace
is $T_4$, is Wilker.  In particular, every Hausdorff space is Wilker.

We will need the following lemma several times in this paper.  The
proof of Theorem~\ref{thm:wilker} is typical of several arguments in
this paper.

\begin{lemma}
  \label{lemma:Gdelta:subspace}
  Let $X$ be a subspace of a topological space $Y$.  For every subset
  $E$ of $X$,
  \begin{enumerate}
  \item $E$ is compact in $X$ if and only if $E$ is compact in $Y$;
  \item if $X$ is upwards-closed in $Y$, then $E$ is saturated in $X$
    if and only if $E$ is saturated in $Y$;
  \item if $X$ is a $G_\delta$ subset of $Y$, then $E$ is $G_\delta$
    in $X$ if and only if $E$ is $G_\delta$ in $Y$.
  \end{enumerate}
\end{lemma}
\begin{proof}
  1. Assume $E$ compact in $X$.  For every open cover
  ${(\widehat U_i)}_{i \in I}$ of $E$ by open subsets of $Y$,
  ${(\widehat U_i \cap X)}_{i \in I}$ is an open cover of $E$ by open
  subsets of $X$.  Hence $E$ has a subcover
  ${(\widehat U_i \cap X)}_{i \in J}$, with $J$ finite, and
  ${(\widehat U_i)}_{i \in J}$ is a finite subcover of the original
  cover of $E$, showing that $E$ is compact in $Y$.

  Conversely, if $E$ is compact in $Y$ (and included in $X$), we
  consider an open cover ${(U_i)}_{i \in I}$ of $E$ by open subsets of
  $X$.  For each $i \in I$, we can write $U_i$ as $\widehat U_i \cap X$
  for some open subset $\widehat U_i$ of $Y$.  Then
  ${(\widehat U_i)}_{i \in I}$ is an open cover of $E$ in $X$.  We
  extract a subcover ${(\widehat U_i)}_{i \in J}$ with $J$ finite.
  Since $E$ is included in $X$, $E$ is included in
  $X \cap \bigcup_{i \in J} \widehat U_i = \bigcup_{i \in J} U_i$.  This
  shows that $E$ is compact in $X$.

  2 follows from the fact that the specialization preordering on $X$
  is the restriction of the specialization preordering on $Y$ to $X$.
  If $E$ is saturated in $X$, then for every $x \in E$ and every $y$
  above $x$ in $Y$, then $y$ is in $X$ since $X$ is saturated, and
  then in $E$ by assumption.  Therefore $E$ is saturated in $Y$.
  Conversely, if $E$ is saturated in $Y$, then for every $x \in E$ and
  every $y$ above $x$ in $X$, $y$ is also above $x$ in $Y$, hence in
  $E$ since $E$ is saturated in $Y$.  This shows that $E$ is saturated
  in $X$.

  3. Since $X$ is $G_\delta$ in $Y$, $X$ is equal to $\bigcap_{n \in
    \nat} W_n$ where each $W_n$ is open in $Y$.

  If $E$ is a $G_\delta$ subset of $X$, say
  $E \eqdef \bigcap_{m \in \nat} U_m$, where each $U_m$ is open in
  $X$, we write $U_m$ as $\widehat U_m \cap X$ for some open subset
  $\widehat U_m$ of $Y$.  It follows that $E$ is equal to
  $\bigcap_{m, n \in \nat} (\widehat U_m \cap W_n)$.  This is a
  countable intersection of open subsets of $Y$, hence a $G_\delta$
  subset.

  Conversely, if $E$ is a $G_\delta$ subset of $Y$, say
  $E \eqdef \bigcap_{m \in \nat} \widehat U_m$, then since $E$ is
  included in $X$, $E$ is also equal to
  $\bigcap_{m \in \nat} (\widehat U_m \cap X)$, showing that $E$ is
  $G_\delta$ in $X$.
\end{proof}

\begin{theorem}
  \label{thm:wilker}
  Every LCS-complete space is Wilker.
\end{theorem}
\begin{proof}
  We start by showing that every locally compact space $Y$ is Wilker,
  and in fact satisfies the following stronger property: $(*)$ for
  every compact saturated subset $Q$ of $Y$, for all open subsets
  $U_1$ and $U_2$ such that $Q \subseteq U_1 \cup U_2$, there are two
  compact saturated subsets $Q_1$ and $Q_2$ such that
  $Q \subseteq \interior {Q_1} \cup \interior {Q_2}$,
  $Q_1 \subseteq U_1$, and $Q_2 \subseteq U_2$.  For each $x \in Q$,
  if $x$ is in $U_1$, then we pick a compact saturated neighborhood
  $Q_x$ of $x$ included in $U_1$, and if $x$ is in $U_2 \diff U_1$,
  then we pick a compact saturated neighborhood $Q'_x$ of $x$ included
  in $U_2$.  From the open cover of $Q$ consisting of the sets
  $\interior {Q_x}$ and $\interior {Q'_x}$, we extract a finite cover.
  Namely, there are a finite set $E_1$ of points of $Q \cap U_1$ and a
  finite set $E_2$ of points of $Q \diff U_1$ (hence in $Q \cap U_2$)
  such that
  $Q \subseteq \bigcup_{x \in E_1} \interior {Q_x} \cup \bigcup_{x \in
    E_2} \interior {Q'_x}$.  We let
  $Q_1 \eqdef \bigcup_{x \in E_1} Q_x$,
  $Q_2 \eqdef \bigcup_{x \in E_2} Q'_x$.
  
  Let $X$ be (homeomorphic to) the intersection
  $\fcap_{n \in \nat} W_n$ of a descending sequence of open subsets of
  a locally compact sober space $Y$.  Let $Q$ be compact saturated in
  $X$, and included in the union of two open subsets $U_1$ and $U_2$
  of $X$.  Let us write $U_1$ as $\widehat U_1 \cap X$ where
  $\widehat U_1$ is open in $Y$, and similarly $U_2$ as
  $\widehat U_2 \cap X$.  By Lemma~\ref{lemma:Gdelta:subspace}, $Q$ is
  compact saturated in $Y$.  By property $(*)$, there are two compact
  saturated subsets $Q_{01}$ and $Q_{02}$ of $Y$ such that
  $Q \subseteq \interior {Q_{01}} \cup \interior {Q_{02}}$,
  $Q_{01} \subseteq \widehat U_1 \cap W_0$,
  $Q_{02} \subseteq \widehat U_2 \cap W_0$.  By $(*)$ again, there are
  two compact saturated subsets $Q_{11}$ and $Q_{12}$ of $Y$ such that
  $Q \subseteq \interior {Q_{11}} \cup \interior {Q_{12}}$,
  $Q_{11} \subseteq \interior {Q_{01}} \cap W_1$,
  $Q_{12} \subseteq \interior {Q_{02}} \cap W_1$.  Continuing this
  way, we obtain two compact saturated subsets $Q_{n1}$ and $Q_{n2}$
  for each $n \in \nat$ such that
  $Q \subseteq \interior {Q_{n1}} \cup \interior {Q_{n2}}$,
  $Q_{(n+1)1} \subseteq \interior {Q_{n1}} \cap W_{n+1}$, and
  $Q_{(n+1)2} \subseteq \interior {Q_{n2}} \cap W_{n+1}$.  Let
  $Q_1 \eqdef \fcap_{n \in \nat} Q_{n1} = \fcap_{n \in \nat} \interior
  {Q_{n1}}$.  This is a filtered intersection of compact saturated
  sets in a well-filtered space, hence is compact saturated.  Since
  $Q_{(n+1)1} \subseteq W_{n+1}$ for every $n \in \nat$, $Q_1$ is
  included in $X$, hence is compact saturated in $X$ by
  Lemma~\ref{lemma:Gdelta:subspace}.  Similarly,
  $Q_2 \eqdef \fcap_{n \in \nat} Q_{n2} = \fcap_{n \in \nat} \interior
  {Q_{n2}}$ is compact saturated in $X$.

  We note that $Q$ is included in $Q_1 \cup Q_2$.  Otherwise, there
  would be a point $x$ in $Q$ and outside both $Q_1$ and $Q_2$, hence
  outside $\interior {Q_{m1}}$ for some $m \in \nat$ and outside
  $\interior {Q_{n1}}$ for some $n \in \nat$, hence outside
  $\interior {Q_{k1}} \cup \interior {Q_{k2}}$ with
  $k \eqdef \max (m, n)$.  That is impossible since
  $Q \subseteq \interior {Q_{k1}} \cup \interior {Q_{k2}}$.

  Finally, $Q_1$ is included in $U_1$ because
  $Q_1 \subseteq Q_{01} \cap X \subseteq \widehat U_1 \cap W_0 \cap X
  = U_1$, and similarly $Q_2$ is included in $U_2$.
\end{proof}

\begin{remark}
  \label{rem:wilker:compactGdelta}
  The proof of Theorem~\ref{thm:wilker} shows that $Q_1$ and $Q_2$ are
  even compact $G_\delta$ subsets of $X$, being obtained as
  $\fcap_{n \in \nat} \interior {Q_{n1}}$, hence also as
  $\fcap_{n \in \nat} \interior {Q_{n1}} \cap X$ (resp.,
  $\fcap_{n \in \nat} \interior {Q_{n2}} \cap X$).  This suggests that
  there are many compact $G_\delta$ sets in every LCS-complete space.
  Note that not all compact saturated sets are $G_\delta$ in general
  LCS-complete spaces: even the upward closures $\upc x$ of single
  points may fail to be $G_\delta$, as Remark~\ref{rem:powI}
  demonstrates.
\end{remark}

\begin{remark}
  \label{rem:ccqm:CS}
  Pursuing Remark~\ref{rem:ccqm:pattern}, every SC-complete space $X$
  is not only LCS-complete, but also \emph{coherent}: the intersection
  of two compact saturated sets $Q_1$, $Q_2$ is compact.  Indeed, let
  $X$ be $G_\delta$ in some stably compact space $Y$; by
  Lemma~\ref{lemma:Gdelta:subspace}, items~1 and~2, $Q_1$ and $Q_2$
  are compact saturated in $Y$, then $Q_1 \cap Q_2$ is compact
  saturated in $Y$ and included in $X$, hence compact in $X$.  This
  implies that there are LCS-complete, and even domain-complete
  spaces, that are not SC-complete: take any non-coherent dcpo, for
  example $\Z^- \cup \{a, b\}$, where $\Z^-$ is the set of negative
  integers with the usual ordering, and $a$ and $b$ are incomparable
  and below $\Z^-$.
\end{remark}

\section{Choquet completeness}
\label{sec:choq-compl-baire}

The \emph{strong Choquet game} on a topological space $X$ is defined
as follows.  There are two players, $\alpha$ and $\beta$.  Player
$\beta$ starts, by picking a point $x_0$ and an open neighborhood
$V_0$ of $x_0$.  Then $\alpha$ must produce a smaller open
neighborhood $U_0$ of $x_0$, i.e., one such that $U_0 \subseteq
V_0$. Player $\beta$ must then produce a new point $x_1$ in $U_0$, and
a new open neighborhood $V_1$ of $x_1$, included in $U_0$, and so on.
An \emph{$\alpha$-history} is a sequence
$x_0, V_0, U_0, x_1, V_1, U_1, x_2, V_2, \cdots, x_n, V_n$ where
$V_0 \supseteq U_0 \supseteq V_1 \supseteq U_1 \supseteq V_2 \supseteq
\ldots \supseteq V_n$ is a decreasing sequence of opens and
$x_0 \in U_0$, $x_1 \in U_1$, $x_2 \in U_2$, \ldots,
$x_{n-1} \in U_{n-1}$, $x_n \in V_n$, $n \in \nat$.  A \emph{strategy}
for $\alpha$ is a map $\sigma$ from $\alpha$-histories to open subsets
$U_n$ with $x_n \in U_n \subseteq V_n$, and defines how $\alpha$ plays
in reaction to $\beta$'s moves.  (For details, see Section~7.6.1 of
\cite{JGL-topology}.)

$X$ is \emph{Choquet-complete} if and only if
$\alpha$ has a winning strategy, meaning that whatever $\beta$ plays,
$\alpha$ has a way of playing such that $\bigcap_{n \in \nat} U_n$
($=\bigcap_{n \in \nat} V_n$) is non-empty.  $X$ is \emph{convergence
  Choquet-complete} if and only if $\alpha$ can always win in such a
way that ${(U_n)}_{n \in \nat}$ is a base of open neighborhoods of
some point.  The latter notion is due to Dorais and Mummert
\cite{DM:choquet}.  We introduce yet another, related notion: a space
is \emph{compactly Choquet-complete} if and only if $\alpha$ can
always win in such that way that ${(U_n)}_{n \in \nat}$ is a base of
open neighborhoods of some non-empty compact saturated set.
We do not assume the strategies to be stationary, that is, the players
have access to all the points $x_n$ and all the open sets $U_n$, $V_n$
played earlier.

The following generalizes \cite[Theorem~4.3]{GLN-lmcs17}, which states
that every continuous complete quasi-metric space is convergence
Choquet-complete in its $d$-Scott topology.
\begin{proposition}
  \label{prop:LCS:Choquet}
  Every domain-complete space is convergence Choquet-complete.  Every
  LCS-complete space is compactly Choquet-complete.
\end{proposition}
\begin{proof}
  Let $X$ be the intersection of a descending sequence
  ${(W_n)}_{n \in \nat}$ of open subsets of $Y$.  Given any open
  subset $U$ of $X$, we write $\widehat U$ for some open subset of $Y$
  such that $\widehat U \cap X = U$ (for example, the largest one).
  % the union $\widehat U$ of all the open subsets of
  % $Y$ whose intersection with $X$ is included in $U$ is the largest
  % open subset of $Y$ such that $\widehat U \cap X = U$.
  % For convenience, we let $U_{-1} \eqdef Y$.

  We first assume that $Y$ is a continuous dcpo.  The proof is a
  variant of \cite[Exercise~7.6.6]{JGL-topology}.  We define
  $\alpha$'s winning strategy so that $U_n$ is of the form
  $\uuarrow y_n \cap X$ for some $y_n \in Y$.  Given the last pair
  $(x_n, V_n)$ played by $\beta$, $x_n$ is the supremum of a directed
  family of elements way-below $x_n$.  One of them will be in
  $\widehat V_n \cap W_n$, and also in $\uuarrow y_{n-1}$ if
  $n \geq 1$, because $x_n \in V_n \subseteq \widehat V_n$,
  $x_n \in X \subseteq W_n$, and (if $n \geq 1$)
  $x_n \in U_{n-1} = \uuarrow y_{n-1} \cap X \subseteq \uuarrow
  y_{n-1}$.  Pick one such element $y_n$ from
  $\widehat V_n \cap W_n \cap \uuarrow y_{n-1}$ (if $n \geq 1$,
  otherwise from $\widehat V_n \cap W_n$), and let $\alpha$ play
  $U_n \eqdef \uuarrow y_n \cap X$, as announced.  Formally, the
  strategy $\sigma$ that we are defining for $\alpha$ is
  $\sigma (x_0, V_0, U_0, x_1, V_1, U_1, x_2, V_2, \cdots, x_n, V_n)
  \eqdef \uuarrow y (x_0, V_0, U_0, x_1, V_1, U_1, x_2, V_2, \cdots,
  x_n, V_n) \cap X$, where
  $y (x_0, V_0, U_0, x_1, V_1, U_1, x_2, V_2, \cdots, x_n, V_n)$ is
  defined by induction on $n$ as a point in $\widehat V_n \cap W_n$,
  and also in
  $\uuarrow y (x_0, V_0, U_0, x_1, V_1, U_1, x_2, V_2, \cdots,
  x_{n-1}, V_{n-1})$ if $n \geq 1$.

  Given any play
  $x_0, V_0, U_0, x_1, V_1, U_1, x_2, V_2, \cdots, x_n, V_n, U_n,
  \cdots$ in the game, let $x \eqdef \sup_{n \in \nat} y_n$ (where
  $y_n \eqdef y (x_0, V_0, U_0, x_1, V_1, U_1, x_2, V_2, \cdots, x_n,
  V_n)$).  This is a directed supremum, since
  $y_0 \ll y_1 \ll \cdots \ll y_n \ll \cdots \ll x$.  Since
  $y_n \in W_n$ for every $n \in \nat$, $x$ is in
  $\bigcap_{n \in \nat} W_n = X$.  For every $n \in \nat$, we have
  $y_n \ll x$, so $x$ is in $U_n = \uuarrow y_n \cap X$.  In order to
  show that ${(U_n)}_{n \in \nat}$ is a base of open neighborhoods of
  $x$ in $X$, let $U$ be any open neighborhood of $x$ in $X$.  Since
  $x = \sup_{n \in \nat} y_n$, some $y_n$ is in $\widehat U$, so
  $U_n = \uuarrow y_n \cap X$ is included in
  $\upc y_n \cap X \subseteq \widehat U \cap X = U$.

  The argument is similar when $Y$ is a locally compact sober space
  instead.  Instead of picking a point $y_n$ in
  $\widehat V_n \cap W_n$ (and in $\uuarrow y_{n-1}$ if $n \geq 1$),
  $\alpha$ now picks a compact saturated subset $Q_n$ whose interior
  contains $x_n$, and included in $\widehat V_n \cap W_n$ (and in
  $\interior {Q_{n-1}}$ if $n \geq 1$), and defines $U_n$ as
  $\interior {Q_n} \cap X$.  This is possible because $Y$ is locally
  compact.  We let $Q \eqdef \bigcap_{n \in \nat} Q_n$.  This is a
  filtered intersection, since
  $Q_0 \supseteq Q_1 \supseteq \cdots \supseteq Q_n \supseteq \cdots$.
  
  Because $Y$ is sober hence well-filtered, $Q$ is compact saturated
  in $Y$.  It is also non-empty: if $Q = \bigcap_{n \in \nat} Q_n$
  were empty, namely, included in $\emptyset$, then $Q_n$ would be
  included in $\emptyset$ by well-filteredness, which is impossible
  since $x_n \in \interior {Q_n}$.  Also, since $Q_n \subseteq W_n$
  for every $n \in \nat$, $Q$ is included in
  $\bigcap_{n \in \nat} W_n = X$.  By
  Lemma~\ref{lemma:Gdelta:subspace}, item~3, $Q$ is a compact
  saturated subset of $X$.

  Since $Q \subseteq Q_{n+1} \subseteq \widehat V_{n+1}$ for every $n
  \in \nat$, we have $Q = Q \cap X \subseteq \widehat V_{n+1} \cap X =
  V_{n+1} \subseteq U_n$.  In order to show that ${(U_n)}_{n \in
    \nat}$ forms a base of open neighborhoods of $Q$ in $X$, let $U$
  be any open neighborhood of $Q$ in $X$.  Then $\widehat U$ contains
  $Q = \bigcap_{n \in \nat} Q_n$, so by well-filteredness some $Q_n$
  is included in $\widehat U$.  Now $U_n = \interior {Q_n} \cap X$ is
  included in $\widehat U \cap X = U$.
\end{proof}
In the case of LCS-complete spaces, notice that
$Q = \bigcap_{n \in \nat} U_n$ is not only compact, but also a
$G_\delta$ subset of $X$.  This again suggests that there are many
compact $G_\delta$ sets in every LCS-complete space, as in
Remark~\ref{rem:wilker:compactGdelta}.

The following---at last---justifies the ``complete'' part in ``LCS-complete''.
\begin{theorem}
  \label{thm:ccqm:metric}
  The metrizable LCS-complete (resp., domain-complete) spaces are the
  completely metrizable spaces.
\end{theorem}
\begin{proof}
  One direction is Corollary~\ref{corl:ccqm:metric}.  Conversely, an
  LCS-complete space is Choquet-complete
  (Proposition~\ref{prop:LCS:Choquet}) and every metrizable
  Choquet-complete space is completely metrizable
  \cite[Corollary~7.6.16]{JGL-topology}.
\end{proof}

\begin{remark}
  \label{rem:nondcomplete}
  There is an LCS-complete but not domain-complete space.  The space
  $\{0, 1\}^I$, where $\{0, 1\}$ is given the discrete topology, is
  compact Hausdorff, hence trivially LCS-complete.  We claim that it
  is not domain-complete if $I$ is uncountable.  In order to show
  that, we first show that: $(*)$ for every point $\vec a$ of
  $\{0, 1\}^I$, every countable family of open neighborhoods
  ${(V_n)}_{n \in \nat}$ of $\vec a$ must be such that
  $\bigcap_{n \in \nat} V_n \neq \{\vec a\}$.  We write $a_i$ for the
  $i$th component of $\vec a$.  For each subset $J$ of $I$, let
  $V_J (\vec a) \eqdef \{\vec b \in \{0, 1\}^I \mid \forall i \in J,
  a_i=b_i\}$; this is a basic open subset of the product topology if
  $J$ is finite.  Since $\vec a \in V_n$, there is a finite subset
  $J_n$ of $I$ such that $\vec a \in V_{J_n} (\vec a) \subseteq V_n$.
  Then $\bigcap_{n \in \nat} V_n$ contains
  $\bigcap_{n \in \nat} V_{J_n} (\vec a) = V_{\bigcup_{n \in \nat}
    J_n} (\vec a)$, which contains uncountably many elements other than
  $\vec a$.  Having established $(*)$, it is clear that no point has a
  countable base of open neighborhoods.  In particular, $\{0, 1\}^I$
  is not convergence Choquet-complete, hence not domain-complete.
\end{remark}

The situation simplifies for countably-based spaces.  We will use the
notion of \emph{supercompact} set: $Q \subseteq X$ is
\emph{supercompact} if and only if for every open cover
${(U_i)}_{i \in I}$ of $Q$, there is an index $i \in I$ such that
$Q \subseteq U_i$.  By \cite[Fact~2.2]{HK:qcont}, the supercompact
subsets of a topological space $X$ are exactly the sets $\upc x$,
$x \in X$.
% Equivalently, the supercompact
% subsets are the non-empty compact subsets $Q$ such that, for all open
% subsets $U$ and $V$ such that $Q \subseteq U \cup V$, $Q$ is included
% in $U$ or in $V$.

% We first need
% the following well-known lemma.
% \begin{lemma}
%   \label{lemma:Q:min}
%   Let $Q$ be a compact saturated subset of a topological space $X$.
%   Every point of $Q$ is above some minimal point of $Q$.
% \end{lemma}
% \begin{proof}
%   Let $x \in Q$.  The family $A$ of all closed subsets $C$ that
%   intersect $Q$ is a dcpo under reverse inclusion: by compactness, if
%   ${(C_i)}_{i \in I}$ is directed in $A$ (namely, filtered for
%   inclusion), $\fcap_{i \in I} C_i$ also intersects $Q$.  By Zorn's
%   Lemma $A$ contains a maximal element above $\dc x$, namely a set
%   $C \in A$ that is minimal for inclusion.  $C$ intersects $Q$, say at
%   $y$.  Since $\dc y$ is also in $A$ and is included in $C$,
%   minimality implies that $C = \dc y$.  Since $C \subseteq \dc x$, $y$
%   is below $x$.  If $y$ was not minimal in $Q$, there would be an
%   element $z$ of $Q$ strictly below $y$, and then $\dc z$ would be a
%   strictly smaller element of $A$: contradiction.
% \end{proof}

\begin{proposition}
  \label{prop:cChoquet:2nd}
  Every countably-based, compactly Choquet-complete space $X$ is
  convergence Choquet-complete.
\end{proposition}
\begin{proof}
  Let $\sigma$ be a strategy for $\alpha$ such that the open sets
  ${(U_n)}_{n \in \nat}$ played by $\alpha$ form a base of open
  neighborhoods of some compact saturated set.  Let also
  ${(B_n)}_{n \in \nat}$ be a countable base of the topology, and let
  us write $B (x, n)$ for
  $\bigcap \{B_i \mid 0\leq i\leq n, x \in B_i\}$.  (In case there is
  no $B_i$ containing $x$ for any $i$, $0 \leq i\leq n$, this is the
  whole of $X$.)  We define a new strategy $\sigma'$ for $\alpha$ by
  using a game stealing argument: when $\beta$ plays $x_n$ and $V_n$,
  $\alpha$ simulates what he would have done if $\beta$ had played
  $x_n$ and $V_n \cap B (x_n, n)$ instead.  Formally, we define
  $\sigma' (x_0, V_0, U_0, x_1, V_1, U_1, x_2, V_2, \cdots, x_n, V_n)
  \eqdef \sigma (x_0, V_0, U_0, x_1, V_1, U_1, x_2, V_2, \cdots, x_n,
  V_n \cap B (x_n, n))$.  Let ${(U'_n)}_{n \in \nat}$ denote the open
  sets played by $\alpha$ using $\sigma'$:
  $U'_n = \sigma (x_0, V_0, U'_0, x_1, V_1, U'_1, x_2, V_2, \cdots,
  x_n, V_n \cap B (x_n, n))$.  Since $X$ is compactly
  Choquet-complete, ${(U'_n)}_{n \in \nat}$ is a countable base of
  open neighborhoods of some non-empty compact saturated set $Q$.  We
  claim that $Q$ is of the form $\upc x$.

  We only need to show that $Q$ is supercompact.  We start by assuming
  that $Q$ is included in the union of two open sets $U$ and $V$, and
  we will show that $Q$ is included in one of them.  We can write $U$
  and $V$ as unions of basic open sets $B_n$, hence by compactness
  there are two finite sets $I$ and $J$ of natural numbers such that
  $Q \subseteq \bigcup_{i \in I \cup J} B_i$,
  $\bigcup_{i \in I} B_i \subseteq U$, and
  $\bigcup_{j \in J} B_j \subseteq V$.  Since ${(U'_n)}_{n \in \nat}$
  is a base of open neighborhoods of $Q$, some $U'_n$ is included in
  $\bigcup_{i \in I \cup J} B_i$.  We choose $n$ higher than every
  element of $I \cup J$.  Since $x_n \in U'_n$, $x_n$ is in
  $\bigcup_{i \in I \cup J} B_i$.  If $x_n$ is in $B_i$ for some
  $i \in I$, then $B (x_n, n)$ is included in $B_i$, and then
  $Q \subseteq U'_n \subseteq B (x_n, n)$ (by the definition of
  $\sigma'$) $\subseteq B_i \subseteq U$.  Otherwise, by a similar
  argument $Q \subseteq V$.

  % Let $x$ be a minimal element of $Q$.  This exists by
  % Lemma~\ref{lemma:Q:min}, since $Q$ is non-empty.  Let
  % $U \eqdef X \diff \dc x$, and $V$ be any open neighborhood of $x$.
  % $Q$ is included in $U \cup V$: every element of
  % $Q \diff U = Q \cap \dc x$ must be both below and above $x$ (by
  % minimality), hence in $V$.  We have seen that $Q$ is included in $U$
  % or in $V$.  Since the first case is impossible, $Q \subseteq V$.
  % Hence $Q$ is included in all the open neighborhoods of $x$,
  % therefore also in their intersection, which is $\upc x$.  Since
  % ${(U'_n)}_{n \in \nat}$ is a countable base of open neighborhoods of
  % $Q$, it is also one of $x$.

  It follows that if $Q$ is included in the union of $n \geq 1$ open
  sets, then it is included in one of them.  Given any open cover
  ${(U_i)}_{i \in I}$ of $Q$, there is a finite subcover
  ${(U_i)}_{i \in J}$ of $Q$.  $J$ is non-empty, since
  $Q \neq \emptyset$.  Hence $Q$ is included in $U_i$ for some
  $i \in J$.  This shows that $Q$ is supercompact.  Hence $Q = \upc x$
  for some $x$.  Since ${(U'_n)}_{n \in \nat}$ is a countable base of
  open neighborhoods of $Q$, it is also one of $x$.
\end{proof}

\begin{theorem}
  \label{thm:qPolish}
  The following are equivalent for a countably-based $T_0$ space $X$:
  \begin{enumerate}
  \item $X$ is domain-complete;
  \item $X$ is LCS-complete;
  \item $X$ is quasi-Polish;
  \item $X$ is compactly Choquet-complete;
  \item $X$ is convergence Choquet-complete.
  \end{enumerate}
\end{theorem}
\begin{proof}
  (iii)$\limp$(i) is by Proposition~\ref{prop:qPolish}, (i)$\limp$(ii)
  is by Proposition~\ref{prop:ccqm:easy}, (ii)$\limp$(iv) is by
  Proposition~\ref{prop:LCS:Choquet}, (iv)$\limp$(v) is by
  Proposition~\ref{prop:cChoquet:2nd}.  Finally, (v)$\limp$(iii) is
  the contents of Theorem~51 of \cite{deBrecht:qPolish}, see also
  \cite[Theorem~11.8]{Chen:qPolish}.
\end{proof}

\begin{remark}
  \label{rem:Q}
  $\rat$, with the usual metric topology, is not quasi-Polish.
  Theorem~\ref{thm:qPolish} implies that it is not LCS-complete
  either.  One can show directly that it is not Choquet-complete, as
  follows.  We fix an enumeration ${(q_n)}_{n \in \nat}$ of $\rat$,
  and we call first element of a non-empty set $A$ the element
  $q_k \in A$ with $k$ least.  At step $0$, $\beta$ plays
  $x_0 \eqdef q_0$, $V_0 \eqdef \rat$.  At step $n+1$, $\beta$ plays
  $V_{n+1}$, defined as $U_n$ minus its first element, and lets
  $x_{n+1}$ be the first element of $V_{n+1}$.  This is possible since
  every non-empty open subset of $\rat$ is infinite.  One checks
  easily that, whatever $\alpha$ plays, $\bigcap_{n \in \nat} V_n$ is
  empty.
\end{remark}

\section{The Baire property}
\label{sec:baire-property}

Every Choquet-complete space is \emph{Baire}
\cite[Theorem~7.6.8]{JGL-topology}, where a Baire space is a space in
which every intersection of countably many dense open sets is dense.
\begin{corollary}
  \label{corl:Baire}
  Every LCS-complete space is Baire.
\end{corollary}
This generalizes Isbell's result that every locally compact sober
space is Baire \cite{Isbell:Baire}.

We will refine this below.  We need to observe the following folklore
result.
\begin{lemma}
  \label{lemma:closed:domain}
  Let $Y$ be a continuous dcpo, and $C$ be a closed subset of $Y$.
  The way-below relation on $C$ is the restriction of the way-below
  relation $\ll$ on $Y$ to $C$.  $C$ is a continuous dcpo, with the
  restriction of the ordering $\leq$ of $Y$ to $C$.
\end{lemma}
\begin{proof}
  First, $C$ is a dcpo under the restriction $\leq_C$ of $\leq$ to
  $C$, and directed suprema are computed as in $Y$.
  
  Let $\ll_C$ denote the way-below relation on $C$.  For all
  $x, y \in C$, if $x \ll y$ (in $Y$) then $x \ll_C y$: every directed
  family of elements of $C$ whose supremum (in $C$, equivalently in
  $Y$) lies above $y$ must contain an element above $x$.

  It follows that $C$ is a continuous dcpo: every element $x$ of $C$
  is the supremum of the directed family of elements $y$ that are
  way-below $x$ in $Y$, and all those elements are in $C$ and
  way-below $x$ in $C$.

  Conversely, we assume $x \ll_C y$, and we consider a directed family
  $D$ in $Y$ whose supremum lies above $y$.  Every continuous dcpo is
  meet-continuous \cite[Theorem~III-2.11]{GHKLMS:contlatt}, meaning
  that if $y \leq \sup D$ for any directed family $D$ in $Y$, then $y$
  is in the Scott-closure of $\dc D \cap \dc y$.  (The theory of
  meet-continuous dcpos is due to Kou, Liu and Luo
  \cite{KLL:meetcont}.)  In the case at hand, $\dc D \cap \dc y$ is
  included in $\dc y$ hence in $C$.  Since $C$ is a continuous dcpo
  and $x \ll_C y$, the set
  $\uuarrow_C x \eqdef \{z \in C \mid x \ll_C z\}$ is Scott-open in
  $C$, and contains $y$.  Then $\uuarrow_C x$ intersects the
  Scott-closure of $\dc D \cap \dc y$, and since it is open, it also
  intersects $\dc D \cap \dc y$ itself, say at $z$.  Then
  $x \ll_C z \leq d$ for some $d \in D$, which implies $x \leq d$.
  Therefore $x \ll y$.
\end{proof}

\begin{proposition}
  \label{prop:ccqm:closed}
  Every $G_\delta$ subset, every closed subset of a domain-complete
  (resp., LCS-complete) space is domain-complete (resp., LCS-complete).
\end{proposition}
\begin{proof}
  Let $X$ be the intersection of a descending sequence
  ${(W_n)}_{n \in \nat}$ of open subsets of $Y$, where $Y$ is a
  continuous dcpo (resp., a locally compact sober space).

  Given any $G_\delta$ subset $A \eqdef \bigcap_{m \in \nat} V_m$ of
  $X$, where each $V_m$ is open in $X$, let $\widehat V_m$ be some
  open subset of $Y$ such that $\widehat V_m \cap X = V_m$.  Then
  $A$ is also equal to the countable intersection $\bigcap_{m, n \in
    \nat} \widehat V_m \cap W_n$, hence $A$ is $G_\delta$ in $Y$.

  Given any closed subset $C$ of $X$, $C$ is the intersection of $X$
  with some closed subset $C'$ of $Y$.  If $Y$ is a continuous dcpo,
  then $C'$ is also a continuous dcpo by
  Lemma~\ref{lemma:closed:domain}.  Then
  $C = \bigcap_{n \in \nat} (C' \cap W_n)$, showing that $C$ is a
  $G_\delta$ subset of $C'$, hence is domain-complete.  If $Y$ is a
  locally compact sober space, $C'$ is sober as a subspace (being the
  equalizer of the indicator function of its complement and of the
  constant $0$ map), and is locally compact: for every $x \in C'$, for
  every open neighborhood $U \cap C'$ of $x$ in $C'$ (where $U$ is
  open in $Y$), there is a compact saturated neighborhood $Q$ of $x$
  in $Y$ included in $U$; then $Q \cap C'$ is a compact saturated
  neighborhood of $x$ in $C'$ included in $U \cap C'$.  Again
  $C = \bigcap_{n \in \nat} (C' \cap W_n)$, showing that $C$ is a
  $G_\delta$ subset of $C'$, hence is LCS-complete.
\end{proof}
A space is \emph{completely Baire} if and only if all its closed
subspaces are Baire.  This is strictly stronger than the Baire
property.  Proposition~\ref{prop:ccqm:closed} and
Corollary~\ref{corl:Baire} together entail the following, which
generalizes the fact that every quasi-Polish space is completely Baire
\cite[beginning of Section~5]{deBrecht:Hurewicz}.
\begin{corollary}
  \label{corl:ccqm:cBaire}
  Every LCS-complete space is completely Baire.
\end{corollary}

\section{Stone duality for domain-complete and LCS-complete spaces}
\label{sec:stone-duality-domain}

There is an adjunction $\Open \dashv \pt$ between $\Topcat$ and the
category of locales $\Loc$---the opposite of the category of frames
$\Frm$.  (See \cite[Section~8.1]{JGL-topology}, for example.)  The
functor $\Open \colon \Topcat \to \Loc$ maps every space $X$ to
$\Open X$, and every continuous map $f$ to the frame homomorphism
$\Open f \colon U \mapsto f^{-1} (U)$.  Conversely,
$\pt \colon \Loc \to \Topcat$ maps every frame $L$ to its set of
completely prime filters, with the topology whose open sets are
$\Open_u \eqdef \{x \in \pt L \mid u \in x\}$, for each $u \in L$.
This adjunction restricts to an adjoint equivalence between the full
subcategories of sober spaces and spatial locales, between the
category of locally compact sober spaces and the opposite of the
category of continuous distributive complete lattices by the
Hofmann-Lawson theorem \cite{HL:lc:sober} (see also
\cite[Theorem~8.3.21]{JGL-topology}), and between the category of
continuous dcpos and the opposite of the category of completely
distributive complete lattices \cite[Theorem~8.3.43]{JGL-topology}.

Let us recall what quotient frames are, following
\cite[Section~3.4]{Heckmann:qPolish}.  More generally, the book by
Picado and Pultr \cite{framesandlocales} is a recommended reference on
frames and locales.  A \emph{congruence preorder} on a frame $L$ is a
transitive relation $\preceq$ such that $u \leq v$ implies
$u \preceq v$ for all $u, v \in L$, $\bigvee_{i \in I} u_i \preceq v$
whenever $u_i \preceq v$ for every $i \in I$, and
$u \preceq \bigwedge_{i=1}^n v_i$ whenever $u \preceq v_i$ for every
$i$, $1\leq i \leq n$.  We can then form the \emph{quotient frame}
$L/{\preceq}$, whose elements are the equivalence classes of $L$ modulo
$\preceq \cap \succeq$.  Given any binary relation $R$ on $L$, there
is a least congruence preorder $\preceq_R$ such that $u \preceq_R v$
for all $(u, v) \in R$.  % (Explicitly,
% one can build $L/{\preceq_R}$ as the subframe of \emph{$R$-saturated}
% elements of $L$, namely those elements $u \in L$ such that
% for all $(a, b) \in R$, for every $v \in L$, $a \wedge v \leq u$
% if and only if $b \wedge v \leq u$
%.)
In particular, for every subset $A$ of $L$, there is a least
congruence preorder $\preceq_A$ such that $\top \preceq_A v$ for every
$v \in A$, where $\top$ is the largest element of $L$.  Using
\cite[Section~11.2]{framesandlocales} for instance, one can check that
$L/{\preceq_A}$ can be equated with the subframe of $L$ consisting of
the \emph{$A$-saturated} elements of $L$, namely those elements
$u \in L$ such that $u = (a \limp u)$ for every $a \in A$, where
$\limp$ is Heyting implication in $L$
($a \limp u \eqdef \bigvee \{b \mid a \wedge b \leq u\}$).

\begin{theorem}
  \label{thm:ccmq:pt}
  The adjunction $\Open \dashv \pt$ restricts to an adjoint
  equivalence between the category of LCS-complete spaces (resp.,
  domain-complete spaces) and the opposite of the category of quotient
  frames $L / {\preceq_A}$, where $A$ is a countable subset of $L$ and
  $L$ is a continuous distributive (resp., completely distributive)
  continuous lattice.
\end{theorem}
\begin{proof}
  We use the following theorem, due to Heckmann
  \cite[Theorem~3.13]{Heckmann:qPolish}: given any completely Baire
  space $Y$, and any countable relation $R \subseteq \Open Y \times
  \Open Y$, the quotient frame $\Open Y / {\preceq_R}$ is spatial, and
  isomorphic to the frame of open sets of $\bigcap_{(U, V) \in R} (Y
  \diff U) \cup V$.

  For every domain-complete (resp., LCS-complete) space $X$, written
  as $\fcap_{n \in \nat} W_n$, where each $W_n$ is open in the
  continuous dcpo (resp., locally sober space) $Y$, $Y$ is itself
  LCS-complete (Proposition~\ref{prop:ccqm:easy}) hence completely
  Baire (Corollary~\ref{corl:ccqm:cBaire}).  It follows from
  Heckmann's theorem that $\Open X$ is isomorphic to
  $\Open Y / {\preceq_A}$ where $A \eqdef \{W_n \mid n \in \nat\}$.
  Therefore $\Open X$ is a quotient frame of a continuous distributive
  (resp., completely distributive) continuous lattice by the countable
  set $A$.

  By Proposition~\ref{prop:ccqm:sober}, every LCS-complete space is
  sober, so the unit
  $x \in X \mapsto \{U \in \Open X \mid x \in U\} \in \pt \Open X$ is
  a homeomorphism \cite[Proposition~8.2.22, Fact~8.2.5]{JGL-topology}.

  In the other direction, let $L$ be a completely distributive (resp.,
  continuous distributive) continuous lattice.  By the Hofmann-Lawson
  theorem, $L$ is isomorphic to the open set lattice of some locally
  compact sober space $Y$.  Without loss of generality, we assume that
  $L = \Open Y$.  As above, $Y$ is LCS-complete hence completely
  Baire.  By Heckmann's theorem, for every countable relation $R$ on
  $L$, $L / {\preceq_R}$ is isomorphic to $\Open X$ where
  $X \eqdef \bigcap_{(U, V) \in R} (Y \diff U) \cup V$.  In
  particular, for any countable subset
  $A \eqdef \{W_n \mid n \in \nat\}$ of $L$, we can equate
  $L / {\preceq_A}$ with $\Open X$ where
  $X \eqdef \bigcap_{n \in \nat} W_n$.  By construction, $X$ is
  domain-complete (resp., LCS-complete).  Finally, the counit
  $U \in L \mapsto \Open_U$ is an isomorphism because $L$ is spatial
  \cite[Proposition~8.1.17]{JGL-topology}.
\end{proof}

\section{Consonance}
\label{sec:consonance}

For a subset $Q$ of a topological space $X$, let $\blacksquare Q$ be the family of open
neighborhoods of $Q$.  A space is \emph{consonant} if and only if,
given any Scott-open family $\mathcal U$ of open sets, and given any
$U \in \mathcal U$, there is a compact saturated set $Q$ such that
$U \in \blacksquare Q \subseteq \mathcal U$.  Equivalently, if and
only if every Scott-open family of opens is a union of sets of the
form $\blacksquare Q$, $Q$ compact saturated.
% The latter generate the
% compact-open topology on the space $\Open X$ of open subsets of $X$,
% if we equate $\Open X$ with the family of continuous maps from $X$ to
% Sierpi\'nski space (namely, $\{0, 1\}$ with the Scott topology of
% $0 < 1$).  Using the same identification, the Isbell topology on
% $\Open X$ coincides with the Scott topology.  The latter Scott
% topology is always finer than the compact-open topology.  A space is
% consonant if and only if the two topologies coincide.

In a locally compact space, every open subset $U$ is the union of the
interiors $\interior Q$ of compact saturated subsets $Q$ of $U$, and
that family is directed.  It follows immediately that every locally
compact space is consonant.  Another class of consonant spaces is
given by the regular \v{C}ech-complete spaces, following Dolecki,
Greco and Lechicki \cite[Theorem~4.1 and footnote~8]{DGL:consonant}.

% \ForAuthors{JGL: if we can prove that all regular \v{C}ech-complete
%   spaces are LCS-complete, then the following generalizes Dolecki et
%   al.}

Consonance is not preserved under the formation of $G_\delta$ subsets
\cite[Proposition~7.3]{DGL:consonant}.  Nonetheless, we have:
\begin{proposition}
  \label{prop:ccqm:consonant}
  Every LCS-complete space is consonant.
\end{proposition}
\begin{proof}
  Let $X$ be the intersection of a descending sequence
  ${(W_n)}_{n \in \nat}$ of open subsets of a locally compact sober
  space $Y$.  Let $\mathcal U$ be a Scott-open family of open subsets
  of $X$, and $U \in \mathcal U$.

  By the definition of the subspace topology, there is an open subset
  $\widehat U$ of $Y$ such that $\widehat U \cap X = U$.  By local
  compactness, $\widehat U \cap W_0$ is the union of the directed
  family of the sets $\interior Q$, where $Q$ ranges over the family
  $\mathcal Q_0$ of compact saturated subsets of
  $\widehat U \cap W_0$.  We have
  $\dcup_{Q \in \mathcal Q_0} \interior Q \cap X = \widehat U \cap
  W_0 \cap X = \widehat U \cap X = U$.  Since $U$ is in $\mathcal U$
  and $\mathcal U$ is Scott-open, $\interior Q \cap X$ is in
  $\mathcal U$ for some $Q \in \mathcal Q_0$.  Let $Q_0$ be this
  compact saturated set $Q$, $\widehat U_0 \eqdef \interior {Q_0}$,
  and $U_0 \eqdef \widehat U_0 \cap X$.  Note that
  $U_0 \in \mathcal U$,
  $\widehat U_0 \subseteq Q_0 \subseteq \widehat U \cap W_0$.

  We do the same thing with $\widehat U_0 \cap W_1$ instead of
  $\widehat U \cap W_0$.  There is a compact saturated subset $Q_1$ of
  $\widehat U_0 \cap W_1$ such that $\interior {Q_1} \cap X$ is in
  $\mathcal U$.  Then, letting $\widehat U_1 \eqdef \interior {Q_1}$,
  $U_1 \eqdef \widehat U_1 \cap X$, we obtain that
  $U_1 \in \mathcal U$,
  $\widehat U_1 \subseteq Q_1 \subseteq \widehat U_0 \cap W_1$.

  Iterating this construction, we obtain for each $n \in \nat$ a
  compact saturated subset $Q_n$ and an open subset $\widehat U_n$ of
  $Y$, and an open subset $U_n$ of $X$ such that $U_n \in \mathcal U$
  for each $n$, and
  $\widehat U_{n+1} \subseteq Q_{n+1} \subseteq \widehat U_n \cap
  W_n$.

  Let $Q \eqdef \bigcap_{n \in \nat} Q_n$.  Since $Y$ is sober hence
  well-filtered, $Q$ is compact saturated in $Y$.

  Since $Q \subseteq \bigcap_{n \in \nat} W_n = X$, $Q$ is a compact
  saturated subset of $Y$ that is included in $X$, hence a compact
  subset of $X$ by Lemma~\ref{lemma:Gdelta:subspace}, item~3.

  We have
  $Q \subseteq Q_0 \subseteq \widehat U \cap W_0 \subseteq \widehat
  U$, and $Q \subseteq X$, so $Q \subseteq \widehat U \cap X = U$.
  Therefore $U \in \blacksquare Q$.

  For every $W \in \blacksquare Q$, write $W$ as the intersection of
  some open subset $\widehat W$ of $Y$ with $X$.  Since
  $Q = \bigcap_{n \in \nat} Q_n \subseteq \widehat W$, by
  well-filteredness some $Q_n$ is included in $\widehat W$.  Hence
  $\widehat U_n \subseteq Q_n \subseteq \widehat W$.  Taking
  intersections with $X$, $U_n \subseteq W$.  Since $U_n$ is in
  $\mathcal U$, so is $W$.
\end{proof}

\begin{remark}
  \label{rem:ccqm:consonant:gdelta}
  In the proof of Proposition~\ref{prop:ccqm:consonant}, $Q$ is a
  $G_\delta$ subset of $X$.  Indeed,
  $\widehat U_{n+1} \subseteq Q_{n+1} \subseteq \widehat U_n \cap W_n$
  for every $n \in \nat$, hence
  $Q = \bigcap_{n \in \nat} Q_n = \bigcap_{n \in \nat} Q_n \cap X =
  \bigcap_{n \in \nat} (\widehat U_n \cap X)$.  Hence we can refine
  Proposition~\ref{prop:ccqm:consonant} to: in an LCS-complete space
  $X$, every Scott-open family $\mathcal U$ of open subsets of $X$ is
  a union of sets $\blacksquare Q$, where the sets $Q$ are compact
  $G_\delta$, not just compact.
\end{remark}

\section{The space $\Lform X$ for $X$ LCS-complete}
\label{sec:space-lform-x}

The topological coproduct of two consonant spaces is not in general
consonant \cite[Example~6.12]{NS:consonant}, whence the need for
the following definition.  Let $n \odot X$ denote the coproduct of $n$
identical copies of $X$.
% I do not know whether,
% given a consonant space $X$, its $n$th \emph{copower} $n \odot X$ (the
% coproduct of $n$ copies of $X$) is consonant, but that seems unlikely.
\begin{definition}[$\odot$-consonant]
  \label{defn:realcons}
  A topological space $X$ is called \emph{$\odot$-consonant} if and only
  if, for every $n \in \nat$, $n \odot X$ is consonant.
\end{definition}
In particular, every $\odot$-consonant space is consonant.
\begin{lemma}
  \label{lemma:ccqm:dotconsonant}
  Every LCS-complete space is $\odot$-consonant.
\end{lemma}
\begin{proof}
  Every topological coproduct of LCS-complete spaces is LCS-complete,
  as we will see in Proposition~\ref{prop:ccqm:coprod}.  For now, let
  us just write the LCS-complete space $X$ as
  $\fcap_{m \in \nat} W_m$, where each $W_m$ is open in the locally
  compact sober space $Y$.  Then $n \odot Y$ is sober, because
  coproducts of sober spaces are sober
  \cite[Lemma~8.4.2]{JGL-topology}.  Since $\Open (n \odot Y)$ is
  isomorphic to ${(\Open Y)}^n$, it is a continuous lattice, so
  $n \odot Y$ is core-compact, hence locally compact.  Then
  $n \odot X$ arises as the $G_\delta$ subset
  $\fcap_{m \in \nat} n \odot W_m$ of $n \odot Y$.  Finally, by
  Proposition~\ref{prop:ccqm:consonant}, $n \odot X$ is consonant.
\end{proof}

Let $[X \to Y]$ denote the space of all continuous maps from $X$ to
$Y$.  A \emph{step function} $g$ from $X$ to $Y$ is a continuous map
whose image is finite.  For every $y$ in the image $\img g$ of $g$,
there is an open neighborhood $V_y$ of $y$ such that
$V_y \cap \img g = \upc y \cap \img g$, namely, that contains only the
elements from $\img g$ that are above $y$.  This is because $\upc y$
is the filtered intersection of the family ${(V_i)}_{i \in I}$ of open
neighborhoods of $y$, and ${(V_i \cap \img g)}_{i \in I}$ is filtered
and finite, hence reaches its infimum.  Then
$U_y \eqdef g^{-1} (\upc y)$ is open because it is also equal to
$g^{-1} (V_y)$.  Moreover, $g$ is the map that sends every element of
$U_y \diff \bigcup_{y' \in \img g, y<y'} U_{y'}$ to $y$.  When $Y$
also has a least element $\bot$, $g$ is then also the pointwise
supremum $\sup_{y \in \img g} U_y \searrow y$, where the elementary
step function $U_y \searrow y$ maps every element of $U_y$ to $y$ and
all other elements to $\bot$.  (The sup is always defined in this
case, whatever $Y$ is provided it has a least element.)  This
generalizes the usual notion of step function.  The following should
be familiar to domain theorists---except that the step functions we
build are not required to be way-below $f$.

A \emph{bounded} family is a set of elements that has an upper bound.
\begin{lemma}
  \label{lemma:step}
  Let $X$ be a topological space, and $Y$ be a continuous poset in
  which every finite bounded family of elements has a least upper
  bound, with its Scott topology.  Every continuous map
  $f \colon X \to Y$ is the pointwise supremum of a directed family of
  step functions.
\end{lemma}
\begin{proof}
  In $Y$, the empty family has a least upper bound, meaning that $Y$
  has a least element $\bot$.  The constant $\bot$ map is a step
  function below $f$.  Given any two step functions $g$, $h$ below
  $f$, let $k$ map every $x \in X$ to the supremum of $g (x)$ and
  $h (x)$, which exists because the family $\{g (x), h (x)\}$ is
  bounded by $f (x)$.  The image of $k$ is clearly finite.  We claim
  that $k$ is continuous.  For every open subset $V$ of $Y$, let $D_V$
  be the set of pairs $(y_1, y_2) \in \img g \times \img h$ such that
  the supremum of $y_1$ and $y_2$ is in $V$.  This is a finite set.
  Then
  $k^{-1} (V) = \bigcup_{(y_1, y_2) \in D_V} g^{-1} (\upc y_1) \cap
  h^{-1} (\upc y_2)$ is open.  Since $k$ is continuous and $\img k$ is
  finite, $k$ is a step function.  This shows that the family
  $\mathcal D$ of step functions pointwise below $f$ is directed.

  For every $x \in X$, and every $y \ll f (x)$ in $Y$, the step
  function $f^{-1} (\uuarrow y) \searrow y$ is in $\mathcal D$, and
  its value at $x$ is $y$.  Since the supremum of all the elements $y$
  way-below $f (x)$ is $f (x)$, $\sup \{g (x) \mid g \in \mathcal D\}
  = f (x)$.
\end{proof}

Given a finite subset $B$ of $Y$, where $Y$ is a poset in which every
finite bounded family $J$ of elements has a least upper bound
$\sup J$, and given any $|B|$-tuple ${(V_y)}_{y \in B}$ of open
subsets of $X$, the notation $\sup_{y \in B} V_y \searrow y$ defines a
step function if and only if every subset $J \subseteq B$ such that
$\bigcap_{y \in J} V_y \neq \emptyset$ is bounded: in that case
$\sup_{y \in B} V_y \searrow y$ maps every point $x \in X$ to
$\sup J$, where $J \eqdef \{y \in B \mid x \in V_y\}$; otherwise, we
say that $\sup_{y \in B} V_y \searrow y$ is \emph{undefined}.
\begin{proposition}
  \label{prop:co=Scott}
  Let $X$ be a $\odot$-consonant space.  Let $Y$ be a continuous poset
  in which every finite bounded family of elements has a least upper
  bound, with its Scott topology.  The compact-open topology on
  $[X \to Y]$ is equal to the Scott topology.
  % \ForAuthors{JGL: Can we remove some assumptions on $Y$, for example the
  %   least upper bound part?}
\end{proposition}
\begin{proof}
  The compact-open topology has subbasic open sets
  $[Q \subseteq V] \eqdef \{f \in [X \to Y] \mid Q \subseteq f^{-1}
  (V)\}$, where $Q$ is compact saturated in $X$ and $V$ is open in
  $Y$.  It is easy to see that $[Q \subseteq V]$ is Scott-open.  In
  the converse direction, let $\mathcal W$ be a Scott-open subset of
  $[X \to Y]$, and $f \in \mathcal W$.  Our task is to find an open
  neighborhood of $f$ in the compact-open topology that is included in
  $\mathcal W$.

  The function $f$ is the pointwise supremum of a directed family of
  step functions, by Lemma~\ref{lemma:step}, hence one of them, say
  $g_0$, is in $\mathcal W$.  We can write $g_0$ as
  $g_0 \eqdef \sup_{y \in B} U_y \searrow y$, with $B$ finite, and
  where each $U_y$ is open.

  Consider the maps $\sup_{y \in B} U_y \searrow z_y$, where
  $z_y \ll y$ for each $y \in B$.  Those maps are defined: for every
  $J \subseteq B$ such that $\bigcap_{y \in J} U_y \neq \emptyset$,
  $\sup J$ exists and is an upper bound of $\{z_y \mid y \in J\}$.
  Explicitly, those maps $\sup_{y \in B} U_y \searrow z_y$ send each
  $x \in X$ to $\sup \{z_y \mid y \in J\}$, where
  $J \eqdef \{y \in B \mid x \in U_y\}$.  Those maps
  form a directed family whose supremum is $g_0$, hence one of them,
  say $g \eqdef \sup_{y \in B} U_y \searrow z_y$, is in $\mathcal W$.

  Let $G$ be the set of subsets $J$ of $B$ such that
  $Z_J \eqdef \{z_y \mid y \in J\}$ is bounded.  For each $J \in G$,
  $Z_J$ has a least upper bound $\sup Z_J$, by assumption.
  % Given any $|B|$-tuple ${(V_y)}_{y \in B}$ of open subsets of $X$, we
  % say that $\sup_{y \in B} V_y \searrow y$ is \emph{defined} if and
  % only if every subset $J$ of $B$ such that $\bigcap_{y \in J} V_y
  % \neq \emptyset$ belongs to $G$; then it is defined as the map that
  % sends every $x \in X$ to $\sup J$ where $J \eqdef \{y \in B \mid x
  % \in V_y\}$; otherwise, $\sup_{y \in B} V_y \searrow y$ is \emph{undefined}.
  Let $\mathcal V$ be the set of $|B|$-tuples ${(V_y)}_{y \in B}$ of
  open subsets of $X$ such that $\sup_{y \in B} V_y \searrow z_y$ is
  undefined or in $\mathcal W$.  Ordering those tuples by
  componentwise inclusion, we claim that $\mathcal V$ is Scott-open.

  We first check that $\mathcal V$ is upwards-closed.  Let
  ${(V_y)}_{y \in B}$ be an element of $\mathcal V$, and
  ${(V'_y)}_{y \in B}$ be a family of open sets such that
  $V_y \subseteq V'_y$ for every $y \in B$.  If
  $\sup_{y \in B} V_y \searrow z_y$ is undefined, then there is a
  subset $J$ of $B$, not in $G$, and such that
  $\bigcap_{y \in J} V_y \neq \emptyset$.  Then
  $\bigcap_{y \in J} V'_y$ is non-empty as well, so
  $\sup_{y \in B} V'_y \searrow z_y$ is undefined, too.  If
  $\sup_{y \in B} V_y \searrow z_y$ is defined, then either
  $\sup_{y \in B} V'_y \searrow z_y$ is undefined, or
  $\sup_{y \in B} V_y \searrow z_y \leq \sup_{y \in B} V'_y \searrow
  z_y$.  In both cases, ${(V'_y)}_{y \in B}$ is in $\mathcal V$.

  Next, let ${(V_y)}_{y \in B}$ be a $|B|$-tuple of open subsets of
  $X$, let $I$ be some indexing set and assume that for every
  $y \in B$, $V_y = \dcup_{i \in I} V_{yi}$, where each $V_{yi}$ is
  open.  If $\sup_{y \in B} V_y \searrow z_y$ is undefined, then there
  is a subset $J$ of $B$, not in $G$, and such that
  $\bigcap_{y \in J} V_y \neq \emptyset$.  We pick an element $x$ from
  $\bigcap_{y \in J} V_y$.  For each $y \in B$, there is an index
  $i \in I$ such that $x \in V_{yi}$, and we can take the same $i$ for
  every $y \in B$ by directedness.  Then
  $\sup_{y \in B} V_{yi} \searrow z_y$ is undefined, hence
  ${(V_{yi})}_{y \in B}$ is in $\mathcal V$.  If instead
  $\sup_{y \in B} V_y \searrow z_y$ is defined, then every map
  $\sup_{y \in B} V_{yi} \searrow z_y$, $i \in I$, is defined, too.
  We claim that
  $\sup_{y \in B} V_y \searrow z_y = \dsup_{i \in I} (\sup_{y \in B}
  V_{yi} \searrow z_y)$.  We fix $x \in X$, and we let
  $J \eqdef \{y \in B \mid x \in V_y\}$.  For every $y \in J$, $x$ is
  in $V_y = \dsup_{i \in I} V_{yi}$ so $x \in V_{yi}$ for some
  $i \in I$.  By directedness, we can choose the same $i$ for every
  $y \in J$.  It follows that
  $(\sup_{y \in B} V_{yi} \searrow z_y) (x) = \sup Z_J = (\sup_{y \in
    B} V_y \searrow z_y) (x)$.  This shows the claim.  Now that we
  know that
  $\sup_{y \in B} V_y \searrow z_y = \dsup_{i \in I} (\sup_{y \in B}
  V_{yi} \searrow z_y)$, and since that is in the Scott-open set
  $\mathcal W$, $\sup_{y \in B} V_{yi} \searrow z_y$ is in
  $\mathcal W$ for some $i \in I$, in particular
  ${(V_{yi})}_{y \in B}$ is in $\mathcal V$.

  We know that $\mathcal V$ is Scott-open.  Moreover, and recalling
  that $g = \sup_{y \in B} U_y \searrow z_y$ is in $\mathcal W$, the
  $|B|$-tuple ${(U_y)}_{y \in B}$ is in $\mathcal V$.  We may equate
  $|B|$-tuples of open subsets with open subsets of $|B| \odot X$, and
  then the compact saturated subsets of $|B| \odot X$ are naturally
  equated with $|B|$-tuples of compact saturated subsets of $X$.
  Since $X$ is $\odot$-consonant, there is a $|B|$-tuple of compact
  saturated subsets $Q_y$, $y \in B$, such that $Q_y \subseteq U_y$
  for every $y \in B$ and such that every $|B|$-tuple ${(V_y)}_{y \in
    B}$ of open sets such that $Q_y \subseteq V_y$ for every $y \in B$
  is in $\mathcal V$.

  Let us consider the compact-open open subset
  $\mathcal W' \eqdef \bigcap_{y \in B} [Q_y \subseteq \uuarrow z_y]$.
  Since $Q_y \subseteq U_y$ for every $y \in B$, $f$ is in $\mathcal
  W'$: for every $y \in B$, for every $x \in Q_y$, $x$ is in $U_y$, so
  $f (x)$, which is larger than or equal to $g_0 (x)$, hence to $y$, is in $\uuarrow
  z_y$.  We claim that $\mathcal W'$ is included in $\mathcal W$.   Let
  $h$ be any element of $\mathcal W'$.  For every $y \in B$, let $V_y
  \eqdef h^{-1} (\uuarrow z_y)$.  Since $h \in [Q_y \subseteq \uuarrow
  z_y]$, $Q_y \subseteq V_y$, so ${(V_y)}_{y \in B}$ is in $\mathcal
  V$, meaning that $\sup_{y \in B} V_y \searrow z_y$ is
  undefined or in $\mathcal W$.  But it cannot be undefined:
  for every $x \in X$, letting $J \eqdef \{y \in B \mid x \in V_y\}$,
  $h (x)$ is an upper bound of $\{z_y \mid y \in J\}$, by the
  definition of $V_y$.  The same argument shows that $\sup_{y \in B}
  V_y \searrow z_y \leq h$.  Since $\sup_{y \in B} V_y \searrow z_y$
  is in $\mathcal W$ and $\mathcal W$ is upwards-closed, $h$ is also
  in $\mathcal W$.
\end{proof}

Let $\Lform X$ denote the space of all continuous maps from $X$ to
$\creals$, the set of extended non-negative real numbers under the
Scott topology.  Those are usually known as the \emph{lower
  semicontinuous} maps from $X$ to $\creal$.  $Y \eqdef \creals$
certainly satisfies the assumptions of
Proposition~\ref{prop:co=Scott}.  Hence:
\begin{corollary}
  \label{corl:co=Scott}
  Let $X$ be a $\odot$-consonant space, for example an LCS-complete
  space.  The compact-open topology on $\Lform X$ is equal to the
  Scott topology on $\Lform X$.  \qed
\end{corollary}

As an application, let us consider Theorem~4.11 of \cite{JGL-mscs16}.
(We will give another application in
Section~\ref{sec:fail-cart-clos}.)  This expresses a homeomorphism
between two kinds of objects.  The first one is the space
$\Prev_{\AN} (X)$ of \emph{sublinear previsions} on $X$, namely
Scott-continuous sublinear maps $F$ from $\Lform X$ to $\creals$,
where sublinear means that $F (ah) = a F (h)$ and
$F (h+h') \leq F (h) + F (h')$ for all $a \in \Rplus$,
$h, h' \in \Lform X$.  $\Prev_{\AN} (X)$ is equipped with the
\emph{weak topology}, whose subbasic open sets are
$[h > r] \eqdef \{F \in \Prev_{\AN} (X) \mid F (h) > r\}$,
$h \in \Lform X$, $r \in \Rplus$.  The second one is
$\Hoare^{cvx} (\Val_\wk (X))$, where $\Val_\wk (X)$ is the space of
continuous valuations on $X$ \cite{jones89,Jones:proba} (more details
in Section~\ref{sec:extens-cont-valu}), or equivalently the space of
\emph{linear} previsions (defined as sublinear previsions, except that
$F (h+h') = F (h) + F (h')$ replaces the inequality
$F (h+h') \leq F (h) + F (h')$), $\Hoare (Y)$ is the space of
non-empty closed subsets of $Y$ with the lower Vietoris topology, and
$\Hoare^{cvx} (Y)$ is the subspace of $\Hoare (Y)$ consisting of its
convex sets.  The already cited Theorem~4.11 of \cite{JGL-mscs16}
states that $\Prev_{\AN} (X)$ and $\Hoare^{cvx} (\Val_\wk (X))$ are
homeomorphic if $\Lform X$ is \emph{locally convex} in its Scott
topology, meaning that every element of $\Lform X$ has a base of
convex open neighborhoods.  The homeomorphism is given by
$r_{\AN} \colon \Hoare^{cvx} (\Val_\wk (X)) \mapsto \Prev_{\AN} (X)$,
$r_{\AN} (C) (h) \eqdef \sup_{\nu \in C} \int_{x \in X} h (x) d\nu$,
and $s_{\AN} \colon \Prev_{\AN} (X) \to \Hoare^{cvx} (\Val_\wk (X))$,
$s_{\AN} (F) \eqdef \{\nu \in \Val_\wk (X) \mid \forall h \in \Lform
X, \int_{x \in X} h (x) d\nu \leq F (h)\}$.  The primary case when
those form a homeomorphism is when $X$ is core-compact.  We have a
second class of spaces where that holds:
\begin{lemma}
  \label{lemma:locconvex}
  For every $\odot$-consonant space, for example an LCS-complete
  space, $\Lform X$ is locally convex in its Scott topology.
\end{lemma}
\begin{proof}
  By Corollary~\ref{corl:co=Scott}, it suffices to show that it is
  locally convex in its compact-open topology, namely that every
  element of $\Lform X$ has a base of convex open neighborhoods.  It
  is routine to show that every basic open $\bigcap_{i=1}^n [Q_i
  \subseteq (a_i, \infty]]$ is convex.
\end{proof}
\begin{corollary}
  \label{corl:homeo:prev}
  For every LCS-complete space, the maps $s_{\AN}$ and $r_{\AN}$
  define a homeomorphism between $\Prev_{\AN} (X)$ and
  $\Hoare^{cvx} (\Val_\wk (X))$.
\end{corollary}
This holds in particular for all continuous complete quasi-metric
spaces in their $d$-Scott topology, in particular for all complete
metric spaces in their open ball topology.

\section{Categorical limits}
\label{sec:limits}

\begin{lemma}
  \label{lemma:ccqm:compact}
  Every domain-complete space is a $G_\delta$ subset of a pointed
  continuous dcpo.  Every LCS-complete space is a $G_\delta$ subset of
  a compact, locally compact and sober space.
\end{lemma}
\begin{proof}
  Let $X$ be the intersection $\bigcap_{n \in \nat} W_n$ of a
  descending family of open subsets of $Y$.  We define the lifting
  $Y_\bot$ of $Y$ as $Y$ plus a fresh element $\bot$ below all others
  (when $Y$ is a dcpo), or as $Y$ plus a fresh element, with open sets
  those of $Y$ plus $Y_\bot$ itself (if $Y$ is a topological space).
  If $Y$ is a continuous dcpo, then so is $Y_\bot$ (it is easy to see
  that $x$ is way-below $y$ in $Y_\bot$ if and only if it is in $Y$,
  or $x=\bot$), and $Y_\bot$ is pointed; if $Y$ is locally compact
  then so is $Y_\bot$ \cite[Exercise~4.8.6]{JGL-topology}; and if $Y$
  is sober then so is $Y_\bot$ \cite[Exercise~8.2.9]{JGL-topology};
  and $Y_\bot$ is compact.  Every open subset of $Y$ is open in
  $Y_\bot$.  Therefore $X$ is the $G_\delta$ subset
  $\bigcap_{n \in \nat} W_n$ of $Y_\bot$.
\end{proof}

\begin{proposition}
  \label{prop:ccqm:prod}
  The topological product of a countable family of domain-complete
  (resp., LCS-complete) spaces is domain-complete (resp.,
  LCS-complete).
\end{proposition}
\begin{proof}
  For each $i \in \nat$, let $X_i$ be the intersection
  $\bigcap_{n \in \nat} W_{in}$ of a descending family of open subsets
  of a continuous dcpo $Y_i$.  We may assume that $Y_i$ is pointed,
  too, by Lemma~\ref{lemma:ccqm:compact}.  The product of pointed
  continuous dcpos is a continuous dcpo, and the Scott topology on the
  product is the product topology
  \cite[Proposition~5.1.56]{JGL-topology}.  Then the topological
  product $\prod_{i \in I} X_i$ arises as the $G_\delta$ subset
  $\bigcap_{n \in \nat} \left(\prod_{i=0}^n W_{i(n-i)} \times
    \prod_{i=n+1}^{+\infty} Y_i\right)$ of $\prod_{i \in I} Y_i$.

  We use a similar argument when each $Y_i$ is locally compact and
  sober instead.  By Lemma~\ref{lemma:ccqm:compact}, we may assume
  that $Y_i$ is compact.  Every product of a family of compact,
  locally compact spaces is (compact and) locally compact
  \cite[Proposition~4.8.10]{JGL-topology}, and every product of sober
  spaces is sober \cite[Theorem~8.4.8]{JGL-topology}.
\end{proof}

\begin{proposition}
  \label{prop:ccqm:eq:no}
  The categories of domain-complete, resp.\ LCS-complete spaces, do
  not have equalizers.
\end{proposition}
\begin{proof}
  Let $X \eqdef \real$, with its usual topology, and
  $Y \eqdef \pow (\real)$, with the Scott topology of inclusion.
  Those are domain-complete spaces.  Define $f, g \colon X \to Y$ by
  $f (x) = (\real \diff \{x\}) \cup \rat$ and $g (x) \eqdef \real$.
  Those are continuous maps: in the case of $f$, this is because
  $f^{-1} (\upc A)$, for every finite $A \subseteq \real$, is the
  complement of the finite set $A \diff \rat$.  The equalizer of $f$
  and $g$ in $\Topcat$ is $\rat$, which is not LCS-complete
  (Remark~\ref{rem:Q}).  That is not enough to show that $f$ and $g$
  do not have an equalizer in the category of LCS-complete (resp.,
  domain-complete) spaces, hence we argue as follows.

  Assume $f$ and $g$ have an equalizer $i \colon Z \to X$ in the
  category of LCS-complete spaces, resp.\ of domain-complete spaces.
  For every $z \in Z$, $f (i (z))=g (i (z))$, so $i (z) \in \rat$.
  Since $i$ is a (regular) mono, and the one-point space $\{*\}$ is
  domain-complete, $i$ is injective: any two distinct points in $Z$
  define two distinct morphisms from $\{*\}$ to $Z$, whose
  compositions with $i$ must be distinct.  If there is a rational
  point $q$ that is not in the image of $i$, then the inclusion map
  $j \colon \{q\} \to X$ is continuous, $\{q\}$ is domain complete,
  $f \circ j = g \circ j$ since $q$ is rational, but $j$ does not
  factor through $i$: contradiction.  Hence the image of $i$ is
  exactly $\rat$.  This allows us to equate $Z$ with $\rat$, with some
  topology, and $i$ with the inclusion map.  Since $i$ is continuous,
  the topology on $Z$ is finer than the usual topology on $\rat$---the
  subspace topology from $\real$.

  We claim that the topology of $Z$ is exactly the usual topology on
  $\rat$.  Let $C$ be a closed subset of $Z$, and let $cl (C)$ be its
  closure in $\real$.  It suffices to show that $cl (C) \cap Z$ is
  included in, hence equal to $C$: this will show that $C$ is closed
  in $\rat$ with its usual topology.  Take any point $x$ from
  $cl (C) \cap Z$.  Since $\real$ is first-countable, there is a
  sequence ${(x_n)}_{n \in \nat}$ of elements of $C$ that converges to
  $x$.  Let us consider $\nat_\infty$, the one-point compactification
  of $\nat$, where $\nat$ is given the discrete topology.  This is a
  compact Hausdorff space, hence it is trivially LCS-complete.  It is
  also countably-based, hence domain-complete (and quasi-Polish) by
  Theorem~\ref{thm:qPolish}.  The map $j \colon \nat_\infty \to X$
  defined by $j (n) \eqdef x_n$, $j (\infty) \eqdef x$ is continuous,
  and $f \circ j = g \circ j$ since the image of $j$ is included in
  $\rat$.  By the universal property of equalizers, $j = i \circ h$
  for some continuous map $h \colon \nat_\infty \to Z$.  We must have
  $h (n) = x_n$ and $h (\infty) = x$.  Since $\infty$ is a limit of
  the numbers $n \in \nat$ in $\nat_\infty$, $x$ must be a limit of
  ${(x_n)}_{n \in \nat}$ in $Z$.  The fact that $C$ is closed in $Z$
  implies that $x$ is in $C$, too, completing the argument.

  Hence $Z$ is $\rat$, and has the same topology.  But this is
  impossible, since $\rat$ is not LCS-complete.
\end{proof}

\begin{remark}
  \label{rem:qPolish:eq}
  In contrast, the category of quasi-Polish spaces has equalizers, and
  they are obtained as in $\Topcat$.  Indeed, for all continuous maps
  $f, g \colon X \to Y$ between two countably-based $T_0$ spaces $X$
  and $Y$, the coequalizer $[f=g] \eqdef \{x \in X \mid f (x)=g (x)\}$
  in $\Topcat$ is a $\bPi^0_2$ subspace of $X$
  \cite[Corollary~10]{deBrecht:qPolish}, and the $\bPi^0_2$ subspaces
  of a quasi-Polish space are exactly its quasi-Polish subspaces
  \cite[Corollary~23]{deBrecht:qPolish}.  We note that those
  properties fail in domain-complete and LCS-complete spaces: the
  singleton subspace $\{I\}$ of $\pow (I)$ (see Remark~\ref{rem:powI})
  is trivially quasi-Polish but not $\bPi^0_2$ in $\pow (I)$, because
  the $\bPi^0_2$ subsets of $\pow (I)$ that contain $I$, the top
  element, must be $G_\delta$ subsets, and we have seen that $\{I\}$
  is not $G_\delta$ in $\pow (I)$.  The reason of the failure is
  deeper: as the following proposition shows, the $\bPi^0_2$ subspaces
  of an LCS-complete space can fail to be LCS-complete.
\end{remark}

Using a named coined by Heckmann \cite{Heckmann:qPolish}, let us call
\emph{UCO subset} of a space $X$ any union of a closed subset with an
open subset.  All UCO subsets are trivially $\bPi^0_2$, and $\bPi^0_2$
subsets are countable intersections of UCO subsets.
\begin{proposition}
  \label{prop:ccmq:UCO}
  The UCO subsets of compact Hausdorff spaces are not in general
  compactly Choquet-complete.  In particular, the UCO subsets of
  LCS-complete spaces are not in general LCS-complete.
\end{proposition}
\begin{proof}
  The second part follows from the first part by
  Fact~\ref{fact:top:complete} and Proposition~\ref{prop:LCS:Choquet}.
  
  Let $X \eqdef [0, 1]^I$, for some uncountable set $I$, and where
  $[0, 1]$ has the usual metric topology.  This is compact Hausdorff.
  Let us fix a closed subset $C$ of $[0, 1]$ with empty interior and
  containing $0$ and at least one other point $a$ (for example,
  $\{0, a\}$), and let $U$ be its complement.  Note that $U$ is dense
  in $[0, 1]$.  $C^I$ is closed in $X$, since its complement is the
  open subset $\bigcup_{i \in I} \pi_i^{-1} (U)$, where
  $\pi_i \colon X \to [0, 1]$ is projection onto coordinate $i$.  Let
  us define $Y$ as the UCO set $\{\vec 0\} \cup (X \diff C^I)$, where
  $\vec 0$ is the point whose coordinates are all $0$.  We claim that
  $Y$ is not compactly Choquet-complete.

  To this end, we assume it is, and we aim for a contradiction.  In
  the strong Choquet game, let $\beta$ play $x_n \eqdef \vec 0$ at
  each round of the game.  Let $U_n$, $n \in \nat$, be the open sets
  played by $\alpha$.  By assumption, $\bigcap_{n \in \nat} U_n$ is a
  compact subset $Q$ of $Y$, hence also of $X$ by
  Lemma~\ref{lemma:Gdelta:subspace}.  For each $n \in \nat$, $U_n$ is
  the intersection of $Y$ with an open neighborhood of $\vec 0$ in
  $X$, and that open neighborhood contains a basic open set
  $\bigcap_{i \in J_n} \pi_i^{-1} (V_{ni})$, where $J_n$ is finite and
  $V_{ni}$ is an open neighborhood of $0$ in $[0, 1]$.  In particular,
  $U_n$ contains $\bigcap_{i \in J_n} \pi_i^{-1} (\{0\}) \cap Y$.  It
  follows that $K = \bigcap_{n \in \nat} U_n$ contains
  $\bigcap_{i \in J} \pi_i^{-1} (\{0\}) \cap Y$, where
  $J \eqdef \bigcup_{n \in \nat} J_n$ is countable.

  Then $I \diff J$ is uncountable, hence non-empty.  Let $k$ be any
  element of $I \diff J$.  Since $\pi_k^{-1} (C)$ contains $C^I$, and
  $Y$ contains $X \diff C^I$, $Y$ contains $\pi_k^{-1} (U)$, so $K$
  contains $\bigcap_{i \in J} \pi_i^{-1} (\{0\}) \cap \pi_k^{-1} (U)$.
  We claim that $K$ must contain
  $\bigcap_{i \in J} \pi_i^{-1} (\{0\})$.  For every element $\vec x$
  in $\bigcap_{i \in J} \pi_i^{-1} (\{0\})$, let $x_k$ be its $k$th
  coordinate, and for every $b \in [0, 1]$, write $\vec x [k:=b]$ for
  the same element with coordinate $k$ changed to $b$.  Since $U$ is
  dense in $[0, 1]$, $x_k$ is the limit of a sequence
  ${(b_n)}_{n \in \nat}$ of elements of $U$.  Then $\vec x [k:=b_n]$,
  $n \in \nat$, form a sequence in $K$ that converges to $\vec x$.
  Since $K$ is compact in a Hausdorff space, hence closed, $\vec x$ is
  in $K$.

  Since $K$ is included in $Y$, $Y$ contains
  $\bigcap_{i \in J} \pi_i^{-1} (\{0\})$, too.  However
  $\vec 0 [k:=a]$ is in the latter, but not in the former since it is
  different from $\vec 0$ and in $C^I$.
\end{proof}

\section{Colimits}
\label{sec:colimits}

\begin{proposition}
  \label{prop:ccqm:coprod}
  The topological coproduct of an arbitrary family of domain-complete
  (resp., LCS-complete) spaces is domain-complete (resp.,
  LCS-complete).
\end{proposition}
\begin{proof}
  For each $i \in I$, let $X_i$ be the intersection
  $\bigcap_{n \in \nat} W_{in}$ of a descending family of open subsets
  of a continuous dcpo $Y_i$.  The coproduct of continuous dcpos is a
  continuous dcpo again, and the Scott topology is the coproduct
  topology \cite[Proposition~5.1.59]{JGL-topology}.  Then we can
  express the coproduct $\coprod_{i \in I} X_i$ as the $G_\delta$
  subset $\bigcap_{n \in \nat} \coprod_{i \in I} W_{in}$ of
  $\coprod_{i \in I} Y_i$.

  When each $Y_i$ is locally compact and sober, we use a similar
  argument, observing the following facts.  First, the compact
  saturated subsets of each $Y_i$ are compact saturated in
  $\coprod_{i \in I} Y_i$.  It follows easily that
  $\coprod_{i \in I} Y_i$ is locally compact.  The coproduct of
  arbitrarily many sober spaces is sober, too
  \cite[Lemma~8.4.2]{JGL-topology}.
\end{proof}

In order to show that coequalizers fail to exist, we make the
following observation.
\begin{lemma}
  \label{lemma:ccqm:countable}
  Every countable compactly Choquet-complete space is
  first-countable, hence countably-based.
\end{lemma}
\begin{proof}
  Let $X$ be countable and compactly Choquet-complete.  Assume that
  $X$ is not first-countable.  There is a point $x$ that has no
  countable base of open neighborhoods.  For each
  $y \in X \diff \upc x$, $X \diff \dc y$ is an open neighborhood of
  $x$, and the intersection of those sets is $\upc x$.
  % their intersection is contained in X - (all y not in \upc x) = \upc
  % x
  % and each X - \dc y contains \upc x (otherwise x \leq y) hence their
  % intersection is \upc x
  Since $X$ is countable, we can therefore write $\upc x$ as the
  intersection of countably many open sets ${(W_n)}_{n \in \nat}$.
  Note that this does \emph{not} say that those open set form a base
  of open neighborhoods: we do not have a contradiction yet.

  In the strong Choquet game, we let $\beta$ play the same point
  $x_n \eqdef x$ at each step.  Initially, $V_0 \eqdef W_0$, and at
  step $n+1$, $\beta$ plays $V_{n+1} \eqdef U_n \cap W_{n+1}$, where
  $U_n$ was the last open set played by $\alpha$.  Note that
  $\fcap_{n \in \nat} V_n \subseteq \fcap_{n \in \nat} W_n = \upc x$,
  while the converse inclusion is obvious.  Since $X$ is compactly
  Choquet-complete, ${(V_n)}_{n \in \nat}$ is a base of open
  neighborhoods of some compact saturated set $Q$, and since
  $\fcap_{n \in \nat} V_n = \upc x$, $Q = \upc x$, and therefore
  ${(V_n)}_{n \in \nat}$ is a base of open neighborhoods of $x$:
  contradiction.

  Finally, every countable first-countable space is countably-based.
\end{proof}

\begin{proposition}
  \label{prop:ccqm:coeq:no}
  The categories of domain-complete, resp.\ LCS-complete spaces, do
  not have coequalizers.
\end{proposition}
\begin{proof}
  Let $\nat_\infty$ be the one-point compactification of $\nat$, the
  latter with its discrete topology.  Let us form the coproduct $Y$ of
  countably many copies of $\nat_\infty$.  Its elements are $(k, n)$
  where $k \in \nat$, $n \in \nat_\infty$.  The \emph{sequential fan}
  is the quotient of $Y$ by the equivalence relation that equates
  every $(k, \infty)$, $k \in \nat$.  This is a known example of a
  countable space that is not countably-based.  That can be realized
  as the coequalizer of $f, g \colon X \to Y$ in $\Topcat$, where $X$
  is $\nat$ with the discrete topology, $f (k) \eqdef (k, \infty)$,
  $g (k) \eqdef (0, \infty)$.  Note that $X$ and $Y$ are
  domain-complete: $X$ is trivially locally compact and sober (since
  Hausdorff), and countably-based, then use Theorem~\ref{thm:qPolish};
  for similar reasons, $\nat_\infty$ is domain-complete, then use
  Proposition~\ref{prop:ccqm:coprod} to conclude that $Y$ is, too.

  Let us assume that $f$ and $g$ have a coequalizer $q \colon Y \to Z$
  in the category of LCS-complete spaces, resp.\ of domain-complete
  spaces.  There is no reason to believe that $Z$ is the sequential
  fan, hence we have to work harder.  There is no reason to believe
  that $q$ is surjective either, since epis in concrete categories may
  fail to be surjective.  However, $q$ is indeed surjective, as we now
  show.  This is done in several steps.  Let $z \in Z$.

  The closure of $z$ is $\dc z$, so
  $\chi_{Z \diff \dc z} \colon Z \to \Sierp$ is continuous, where
  $\Sierp \eqdef \{0 < 1\}$ is Sierpi\'nski space---trivially a
  continuous dcpo, hence a domain-complete space.  Let $\mathbf 1$ be
  the constant map equal to $1 \in \Sierp$.  If $\dc z$ did not
  intersect the image of $q$, then $\chi_{Z \diff \dc z} \circ q$
  would be equal to $\mathbf 1 \circ q$, although
  $\chi_{Z \diff \dc z} \neq \mathbf 1$, and that is impossible since
  $q$ is epi.  Therefore $\dc z$ intersects the image of $q$.

  Imagine that there were two distinct points $q (k_1, n_1)$,
  $q (k_2, n_2)$ in $\dc z$.  In particular, $(k_1, n_1)$ and
  $(k_2, n_2)$ are distinct.  Also, not both $n_1$ and $n_2$ are equal
  to $\infty$, since otherwise
  $q (k_1, n_1) = q (k_1, \infty) = q (f (k_1)) = q (g (k_1)) = q (g
  (k_2))$ (since $g$ is a constant map)
  $= q (f (k_2)) = q (k_2, \infty) = q (k_2, n_2)$.  Without loss of
  generality, let us say that $n_1 \neq \infty$.  We consider the map
  $\chi_{\{(k_1, n_1)\}} \colon Y \to \{0, 1\}$, where $\{0, 1\}$ has
  the discrete topology (and is a continuous dcpo with the equality
  ordering, hence domain-complete).  Observe that this is a continuous
  map, owing to the fact that $n_1 \neq \infty$.  Since
  $\chi_{\{(k_1, n_1)\}} \circ f = \chi_{\{(k_1, n_1)\}} \circ g$
  ($=0$), $\chi_{\{(k_1, n_1)\}} = h \circ q$ for some unique
  continuous map $h \colon Z \to \{0, 1\}$, by the definition of a
  coequalizer.  Then $h (q (k_1, n_1)) = 1$, while
  $h (q (k_2, n_2)) = 0$, but since $h$ is continuous it must be
  monotonic with respect to the underlying specialization orderings,
  so $q (k_1, n_1) \leq z$ implies $h (q (k_1, n_1)) = h (z)$, and
  similarly $h (q (k_2, n_2)) = h (z)$.  This would imply
  $1 = h (z) = 0$, a contradiction.  Hence there is exactly one point
  $z' \leq z$ in the image of $q$.

  Consider the two maps
  $\chi_{Z \diff \dc z'}, \chi_{Z \diff \dc z} \colon Z \to \Sierp$.
  For every $(k, n) \in Y$, if $\chi_{Z \diff \dc z'} (q (k, n))=0$,
  then $q (k, n) \leq z' \leq z$, so
  $\chi_{Z \diff \dc z} (q (k, n))=0$; conversely, if
  $\chi_{Z \diff \dc z} (q (k, n))=0$, then $q (k, n)$ is below $z$
  and is therefore the unique point $z' \leq z$ in the image of $q$,
  so
  $\chi_{Z \diff \dc z'} (q (k, n)) = \chi_{Z \diff \dc z'} (z') = 0$.
  Hence we have two morphisms which yield the same map when composed
  with $q$.  Since $q$ is epi, they must be equal.  It follows that
  $\dc z = \dc z'$, and since $Z$ is $T_0$ (since sober, see
  Proposition~\ref{prop:ccqm:sober}), $z=z'$.  Therefore $z$ is in the
  image of $q$.  This completes the proof that $q$ is surjective.

  Since $q$ is surjective, and $Y$ is countable, so is $Z$.  By
  Lemma~\ref{lemma:ccqm:countable}, $Z$ is first-countable.  Let
  $\omega \eqdef q (0, \infty)$.  For every $k \in \nat$,
  $q (k, \infty) = q (f (k)) = q (g (k)) = \omega$.  Let
  ${(B_k)}_{k \in \nat}$ be a countable base of open neighborhoods of
  $\omega$ in $Z$.  For each $k \in \nat$, since
  ${(k, n)}_{n \in \nat}$ converges to $(k, \infty)$ in $Y$,
  ${(q (k, n))}_{n \in \nat}$ converges to $\omega$, so $q (k, n)$ is
  in $B_k$ for $n$ large enough.  Let us fix some $n_k \in \nat$ such
  that $(k, n) \in q^{-1} (B_k)$ for every $n \geq n_k$.  Let
  $h \colon Y \to \{0, 1\}$ map every point $(k, n)$ to $0$ if
  $n \leq n_k$, to $1$ if $n > n_k$.  This is continuous,
  $h \circ f = \mathbf 1 = h \circ g$, so $h = h' \circ q$ for some
  unique continuous map $h' \colon Z \to \{0, 1\}$.  Since
  $h (0, \infty)=1$, $h' (\omega)=1$.  By definition of a base, the
  open neighborhood ${h'}^{-1} (\{1\})$ of $\omega$ contains some
  $B_k$.  Recall that $(k, n_k)$ is in $q^{-1} (B_k)$, hence also in
  $q^{-1} ({h'}^{-1} (\{1\})) = h^{-1} (\{1\})$, so $h (k, n_k)=1$.
  However, by definition of $h$, $h (k, n_k)=0$.  We reach a
  contradiction, so the coequalizer of $f$ and $g$ does not exist.
\end{proof}

\begin{remark}
  \label{rem:qpolish:coeq:no}
  The same proof shows that the category of quasi-Polish spaces does
  not have coequalizers.
\end{remark}

\section{The failure of Cartesian closure}
\label{sec:fail-cart-clos}

\begin{proposition}
  \label{prop:ccqm:ccc:no}
  In the category of domain-complete, resp.\ LCS-complete spaces,
  every exponentiable object is locally compact sober.  The categories
  of domain-complete, resp.\ LCS-complete spaces, are not
  Cartesian-closed.
\end{proposition}
\begin{proof}
  Let $X$ be an exponentiable object in any of those categories.  By
  \cite[Theorem~5.5.1]{JGL-topology}, in any full subcategory of
  $\Topcat$ with finite products and containing $1 \eqdef \{*\}$ as an
  object, and up to a unique isomorphism, the exponential $Y^X$ of two
  objects $X$, $Y$ is the space $[X \to Y]$ of all continuous maps
  from $X$ to $Y$, with some uniquely determined topology.  We take
  $Y \eqdef \Sierp$.  Then $[X \to Y]$ can be equated with the lattice
  $\Open X$ of open subsets of $X$.  The application map from
  $[X \to Y] \times X$ to $Y$ is continuous, and notice that product
  $\times$ here is just topological product
  (Proposition~\ref{prop:ccqm:coprod}).  It follows that the graph
  $(\in)$ of the membership relation on the topological product
  $X \times \Open X$ is open.  By \cite[Exercise~5.2.7]{JGL-topology},
  this happens if and only if $X$ is core-compact.  Since $X$ is also
  sober (Proposition~\ref{prop:ccqm:sober}), and sober core-compact
  spaces are locally compact \cite[Theorem~8.3.10]{JGL-topology}, $X$
  must be locally compact.  Now take any non-locally compact
  LCS-complete space, for example Baire space $\nat^\nat$, which is
  Polish but not locally compact.
\end{proof}

\begin{remark}
  \label{rem:polish:ccc:no}
  The same proof shows that the category of quasi-Polish spaces is not
  Cartesian-closed.  A similar proof, with $[0, 1]$ replacing
  $\Sierp$, would show that the category of Polish spaces is not
  Cartesian-closed, using Arens' Theorem \cite{Arens:transform} (see
  also \cite[Exercise~6.7.25]{JGL-topology}): the completely regular
  Hausdorff spaces that are exponentiable in the category of Hausdorff
  spaces are exactly the locally compact Hausdorff spaces.
\end{remark}

We can be more precise on the subject of quasi-Polish spaces.
\begin{theorem}
  \label{thm:qpolish:exp}
  The exponentiable objects $X$ in the category of quasi-Polish spaces
  are the locally compact quasi-Polish spaces, i.e., the
  countably-based locally compact sober spaces.  For every
  quasi-Polish space $Y$, the exponential object is $[X \to Y]$ with the
  compact-open topology.
\end{theorem}
\begin{proof}
  We first note that every quasi-Polish space is sober and
  countably-based, and that conversely every countably-based locally
  compact sober is quasi-Polish \cite[Theorem~44]{deBrecht:qPolish}.
  
  Assume $X$ is locally compact quasi-Polish, and $Y$ is quasi-Polish.
  The only thing we must show is that $[X \to Y]$, with the
  compact-open topology, is quasi-Polish.  Indeed, the application map
  from $[X \to Y] \times X$ to $Y$ will automatically be continuous,
  and the currification $z \mapsto (x \mapsto f (z, x))$ of every
  continuous map $f \colon Z \times X \to Y$ will be continuous from
  $Z$ to $[X \to Y]$, because $X$ is exponentiable in $\Topcat$
  \cite[Theorem~5.4.4]{JGL-topology} and the exponential object is
  $[X \to Y]$, with the compact-open topology, owing to the fact that
  $X$ is locally compact \cite[Exercise~5.4.8]{JGL-topology}.

  Up to homeomorphism $Y$ is a $\bPi^0_2$ subspace of $\pow (\nat)$
  \cite[Corollary~24]{deBrecht:qPolish}.  Hence write $Y$ as
  $\{z \in \pow (\nat) \mid \forall n \in \nat, z \in U_n \limp z \in
  V_n\}$, where $U_n$ and $V_n$ are open.  As in the proof of
  Proposition~\ref{prop:ccqm:sober}, we define
  $f,g \colon \pow (\nat) \to \pow (\nat)$ by
  $f (z) \eqdef \{n \in \nat \mid z \in U_n\}$,
  $g (z) \eqdef \{n \in \nat \mid z \in U_n \cap V_n\}$.
  % $f (z) \eqdef {(\chi_{U_n} (z))}_{n \in \nat}$,
  % $g (z) \eqdef {(\chi_{U_n \cap V_n} (z))}_{n \in \nat}$.
  The equalizer of $f$ and $g$ in $\Topcat$ is $Y$.

  Since $X$ is exponentiable, the exponentiation functor $\_^X$ on
  $\Topcat$ is well-defined and is right adjoint to the product
  functor $\_ \times X$ on $\Topcat$.  Since right adjoints preserve
  limits, in particular equalizers, $Y^X$ is the equalizer of the maps
  $f^X, g^X \colon {(\pow (\nat))}^X \to {(\pow (\nat))}^X$.  Since
  $X$ is locally compact, we know that $Y^X$ is $[X \to Y]$ with the
  compact-open topology (see \cite[Exercise~5.4.11]{JGL-topology} for
  example).  % It remains to show that $Y^X$ is quasi-Polish.

  Similarly, ${(\pow (\nat))}^X = [X \to \pow (\nat)]$ with the
  compact-open topology.  Recall that every quasi-Polish space is
  LCS-complete hence $\odot$-consonant
  (Lemma~\ref{lemma:ccqm:dotconsonant}), and $\pow (\nat)$ is an
  algebraic complete lattice.  By Proposition~\ref{prop:co=Scott}, the
  compact-open topology on $[X \to \pow (\nat)]$ is the Scott
  topology.

  We now use \cite[Proposition~II-4.6]{GHKLMS:contlatt}, which says
  that if $X$ is core-compact and $L$ (here $\pow (\nat)$) is an
  injective $T_0$ space (i.e., a continuous complete lattice by
  \cite[Theorem~II-3.8]{GHKLMS:contlatt}), then $[X \to L]$ is a
  continuous complete lattice.  We claim that $[X \to \pow (\nat)]$ is
  countably-based.  This follows froms
  \cite[Corollary~III-4.10]{GHKLMS:contlatt}, which says that when $X$
  is a $T_0$ core-compact space and $L$ is a continuous lattice such
  that $w \eqdef \max (w (X), w (L))$ is infinite ($w (L)$ is the
  minimal cardinality of a basis of $L$, and is $\omega$ in our case;
  $w (X)$ is the \emph{weight} of $X$, namely the minimal cardinality
  of a base of $X$, and is less than or equal to $\omega$, by
  assumption), then $w ([X \to L]) \leq w (\Open [X \to L]) \leq w$.
  We have shown that ${(\pow (\nat))}^X = [X \to \pow (\nat)]$ is a
  countably-based continuous dcpo, hence an $\omega$-continuous dcpo
  by a result of Norberg \cite[Proposition~3.1]{Norberg:randomsets}
  (see also \cite[Lemma~7.7.13]{JGL-topology}), hence a quasi-Polish
  space.

  The equalizer (in $\Topcat$) of two continuous maps between
  countably-based $T_0$ spaces is a $\bPi^0_2$ subspace of the source
  space \cite[Corollary~10]{deBrecht:qPolish}.  Hence $Y^X$ is
  $\bPi^0_2$ in ${(\pow (\nat))}^X$.  Since the $\bPi^0_2$ subspaces
  of a quasi-Polish space are exactly its quasi-Polish subspaces
  \cite[Corollary~23]{deBrecht:qPolish}, $Y^X=[X \to Y]$ is quasi-Polish.
\end{proof}

% \ForAuthors{JGL: are those categories closed under projective limits
%   indexed by $\nat$?}

\section{Compact subsets of LCS-complete spaces}
\label{sec:compact-subsets-lcs}

% \ForAuthors{JGL: should we keep this section?}

A well-known theorem due to Hausdorff states that, in a complete
metric space, a subset is compact if and only if it is closed and
precompact, where precompact means that for every $\epsilon > 0$, the
subset can be covered by finitely many open balls of radius
$\epsilon$.  An immediate consequence is as follows.  Build a finite
union $A_0$ of closed balls of radii at most $1$.  Then build a finite
union $A_1$ of closed balls of radii at most $1/2$ included in $A_0$,
then a finite union $A_2$ of closed balls of radii at most $1/4$
included in $A_1$, and so on.  Then $\fcap_{n \in \nat} A_n$ is
compact.  (That argument is the key to showing that every bounded
measure on a Polish space is tight, for example.)  We show that a
similar construction works in LCS-complete spaces.

In this section, we fix a presentation of an LCS-complete space $X$ as
$\ker \mu$ for some continuous map $\mu \colon Y \to \creal^{op}$, $Y$
locally compact sober (see Remark~\ref{rem:measurement}).  Replacing
$\mu$ by $\frac 2 \pi \arctan \circ \mu$, we may assume that $\mu$
takes its values in $[0, 1]$.

For every non-empty compact saturated subset $Q$ of $Y$, the image
$\mu [Q]$ of $Q$ by $\mu$ is compact in $[0, 1]^{op}$, hence has a
largest element.  Let us call that largest value the \emph{radius}
$r (Q)$ of $Q$.  Note that this depends not just on $Y$, but also on
$\mu$.  Note also that $r (\upc y) = \mu (y)$ for every $y \in Y$, and
that $r (\bigcup_{i=1}^n Q_i) = \max \{r (Q_i) \mid 1\leq i\leq n\}$.

\begin{remark}
  \label{rem:radius}.
  The name ``radius'' comes from the following observation.  In the
  special case where $Y = \mathbf B (X, d)$ for some continuous
  complete quasi-metric space $X, d$, we may define
  $\mu (x, r) \eqdef r$, and in that case the radius of $Q$ is
  $\max \{r \mid x \in X, (x, r) \in Q\}$.
\end{remark}

\begin{lemma}
  \label{lemma:Xd:compact:1}
  Let $X$, $Y$, $\mu$ be as above.  For every filtered family
  ${(Q_i)}_{i \in I}$ of non-empty compact saturated subsets of $Y$
  such that $\inf_{i \in I} r (Q_i)=0$, $\fcap_{i \in I} Q_i$ is a
  non-empty compact saturated subset of $X$.
\end{lemma}
\begin{proof}
  Since $Y$ is sober hence well-filtered,
  $Q \eqdef \fcap_{i \in I} Q_i$ is a non-empty compact saturated
  subset of $Y$.  We show that $Q$ is included in $X$ by showing that,
  for every $y \in Q$, for every $\epsilon > 0$, $\mu (y) < \epsilon$.
  Indeed, since $\inf_{i \in I} r (Q_i)=0$, we can find an index
  $i \in I$ such that $r (Q_i) < \epsilon$.  Then
  $\mu (y) \leq r (Q_i)$, by definition of radii, and since
  $y \in Q_i$.

  Hence $Q$ is compact saturated in $Y$, and included in $X$, hence it
  is compact saturated in $X$, by Lemma~\ref{lemma:Gdelta:subspace},
  items~1 and 2.
\end{proof}

\begin{lemma}
  \label{lemma:cont:Q}
  Let $X$, $Y$, $\mu$ be as above.  For every non-empty compact
  saturated subset $Q$ of $X$, for every open neighborhood $U$ of $Q$
  in $Y$, for every $\epsilon > 0$, there is a non-empty compact
  saturated subset $Q'$ of $Y$ such that
  $Q \subseteq \interior {Q'} \subseteq Q' \subseteq U$ and
  $r (Q') < \epsilon$.

  If $Y$ is a continuous dcpo, we can even take $Q'$ of the form
  $\upc A$ for some non-empty finite set $A = \{y_1, \cdots, y_n\}$,
  where $\mu (y_i) < \epsilon$ for every $i$.
\end{lemma}
\begin{proof}
  $U \cap \mu^{-1} ([0, \epsilon))$ is open, hence by local
  compactness it is the directed union of sets of the form
  $\interior {Q'}$, where each $Q'$ is compact saturated and included
  in $U \cap \mu^{-1} ([0, \epsilon))$.  The open sets
  $\interior {Q'}$ form a cover of $Q$, which is compact saturated in
  $Y$ by Lemma~\ref{lemma:Gdelta:subspace}, items~1 and 2, so some
  $Q'$ as above is such that $Q \subseteq \interior {Q'}$.  By
  construction, $Q' \subseteq U$.  Also, $r (Q') < \epsilon$ because
  $Q' \subseteq \mu^{-1} ([0, \epsilon))$.

  We prove the second part of the lemma in the more general case where
  $Y$ is quasi-continuous.  Then $Y$ is locally finitary compact
  \cite[Exercise~5.2.31]{JGL-topology}, meaning that we can replay the
  above argument with $Q'$ of the form $\upc A$ for $A$ finite.
\end{proof}

% Let us say that a subset of $Y$ is \emph{$\epsilon$-fine} if and only
% if it is included in a finite union of interiors of compact saturated
% subsets $Q_1$, \ldots, $Q_n$ of radii less than $\epsilon$.
% For every $Q \subseteq Y$, and every $\epsilon > 0$, let us call
% \emph{$\epsilon$-fine cover} $C$ any finite collection of compact
% saturated subsets $Q_1$, \ldots, $Q_n$ of $Y$ of radii $<\epsilon$ and
% such that $Q \subseteq \bigcup_{i=1}^n \interior {Q_i}$.
% \begin{proposition}
%   \label{prop:Xd:compact}
%   Let $X$, $Y$, $\mu$ be as above.  
%   For every compact saturated subset
%   $Q$ of $X$, there is a filtered family

%   ${(Q_{ni})}_{i \in I_n}$, $n \in \nat$, of compact saturated of $Y$
%   such that:
%   \begin{enumerate}
%   \item $r (Q_{ni}) < 1/2^n$ for all $n \in \nat$ and $i \in I$;
%   \item $Q \subseteq \interior {A_n}$ for every $n \in \nat$, where
%     $A_n \eqdef \bigcup_{i \in I_n} Q_{ni}$;
%   \item $A_{n+1} \subseteq \interior {A_n}$ for every $n \in \nat$;
%   \item $Q = \fcap_{n \in \nat} \interior {A_n} = \fcap_{n \in \nat} A_n$.
%   \end{enumerate}
% \end{proposition}

\begin{theorem}
  \label{thm:Xd:compact}
  Let $X$, $Y$, $\mu$ be as above.  The non-empty compact saturated
  subsets of $X$ are exactly the filtered intersections
  $\fcap_{i \in I} Q_i$ of (interiors of) non-empty compact saturated
  subsets $Q_i$ of $Y$ such that $\inf_{i \in I} r (Q_i)=0$.
  Moreover, we can choose that filtered intersection to be equal to
  $\fcap_{i \in I} \interior {Q_i}$.

  When $Y$ is a continuous dcpo, we can even take $Q_i$ of the form
  $\upc A_i$, $A_i$ finite.
\end{theorem}
\begin{proof}
  One direction is Lemma~\ref{lemma:Xd:compact:1}.  Conversely, let
  $Q$ be compact saturated in $X$, and let ${(Q_i)}_{i \in I}$ be the
  family of compact saturated subsets of $Y$ such that
  $Q \subseteq \interior {Q_i}$ (respectively, only those of the form
  $\upc A_i$ with $A_i$ finite, if $Y$ is a continuous dcpo).  By
  Lemma~\ref{lemma:cont:Q} with $U \eqdef Y$, for every $\epsilon > 0$
  there is an index $i \in I$ such that $r (Q_i) < \epsilon$, so
  $\inf_{i \in I} r (Q_i)=0$.  This also shows that the family is
  non-empty.  For any two elements $Q_i$, $Q_j$ of the family, we
  apply Lemma~\ref{lemma:cont:Q} with
  $U \eqdef \interior {Q_i} \cap \interior {Q_j}$ (and $\epsilon$
  arbitrary), and we obtain an element $Q_k$ such that
  $Q_k \subseteq \interior {Q_i} \cap \interior {Q_j}$.  This shows
  that the family is filtered.

  For every open neighborhood $U$ of $Q$ in $Y$,
  Lemma~\ref{lemma:cont:Q} (again) shows the existence of an index
  $i \in I$ such that $Q_i \subseteq U$.  Therefore
  $Q = \fcap_{i \in I} Q_i$.  Finally, $Q \subseteq \interior {Q_i}$
  for every $i \in I$, so
  $Q \subseteq \fcap_{i \in I} \interior {Q_i} \subseteq \fcap_{i \in
    I} Q_i = Q$, so all the terms involved are equal.
\end{proof}

In particular, if $X, d$ is a continuous complete quasi-metric space,
and taking $Y \eqdef \mathbf B (X, d)$ and $\mu (x, r) \eqdef r$, then
the compact saturated subsets of $X$ (in its $d$-Scott topology) are
exactly the filtered intersections of sets $C_i \eqdef Q_i \cap X$.
For each $i$, we can take $Q_i$ of the form
$\upc \{(x_1, r_1), \ldots, (x_n, r_n)\}$ where
$r (Q_i) = \max \{r_1, \cdots, r_n\}$ is arbitrarily small.  Then the
sets $C_i$ are easily seen to be finite unions of closed balls
$B_{x_i, \leq r_i}$ of arbitrarily small radius.  That explains the
connection with Hausdorff's theorem cited earlier.  Note, however,
that closed balls are in general not closed (except when $X, d$ is
metric), and need not be compact either.

\section{Extensions of continuous valuations}
\label{sec:extens-cont-valu}

Continuous valuations were introduced in \cite{jones89,Jones:proba}.
As far as measure theory is concerned, we refer the reader to any
standard reference, such as \cite{Billingsley:probmes}.

A \emph{valuation} $\nu$ on a space $X$ is a map from the lattice of
open subsets $\Open X$ of $X$ to $\creal$ that is \emph{strict}
($\nu (\emptyset)=0$) and \emph{modular}
($\nu (U) + \nu (V) = \nu (U \cup V) + \nu (U \cap V)$).  A
\emph{continuous valuation} is additionally Scott-continuous.  Every
continuous valuation $\nu$ defines a linear prevision $G$ by
$G (h) \eqdef \int_{x \in X} h (x) d\nu$, and conversely any linear
prevision defines a continuous valuation $\nu$ by
$\nu (U) \eqdef G (\chi_U)$, where $\chi_U$ is the characteristic map
of $U$.

Any pointwise directed supremum of continuous valuations is a
continuous valuation again.

A continuous valuation $\nu$ is \emph{locally finite} if and only if
every point has an open neighborhood $U$ such that $\nu (U) < \infty$.
It is \emph{bounded} if and only if $\nu (X) < \infty$.  Let
$\mathcal A (\Open X)$ be the smallest Boolean algebra of subsets of
$X$ containing $\Open X$.  The elements of $\mathcal A (\Open X)$ are
the finite disjoint unions of \emph{crescents}, where a crescent is a
difference $U \diff V$ of two open sets.  The Smiley-Horn-Tarski
theorem \cite{smiley44,HornTarski48:ext} states that every bounded
valuation extends to a unique strict modular map from
$\mathcal A (\Open X)$ to $\Rplus$.

Given any open set $U$, $\nu_{|U}$ is the continuous valuation defined
by $\nu_{|U} (V) \eqdef \nu (U \cap V)$; that is bounded if and only
if $\nu (U) < \infty$.

Let us write $\Borel X$ for the Borel $\sigma$-algebra of $X$.  A
measure on $X$ is a $\sigma$-additive map from $\Borel X$ to $\creal$,
or equivalently a strict, modular and $\omega$-continuous map from
$\Borel X$ to $\creal$.  The latter makes it clear that the pointwise
directed supremum of a family (even uncountable) of measures is a
measure.

We will use the following standard fact, which we shall call
\emph{Kolmogorov's criterion}: given a bounded measure $\mu$, and a
descending sequence ${(W_n)}_{n \in \nat}$ of Borel sets,
$\mu (\fcap_{n \in \nat} W_n) = \inf_{n \in \nat} \mu (W_n)$.  We will
also use the following: any two bounded measures that agree on
$\Open X$ agree on the whole of $\Borel X$.
% We will
% also use the following consequence of the so-called
% $\lambda\pi$-theorem \cite{Sierp:lambdapi}.  A \emph{$\pi$-system} if
% a family of subsets of $X$ that is closed under finite intersections.
% Then every two bounded measures that agree on a $\pi$-system
% $\Pi \subseteq \Borel X$ also agree on the smallest $\sigma$-algebra
% containing $\Pi$.

If $X$ is countably-based, or more generally if $X$ is
\emph{hereditarily Lindel\"of} (viz., every directed family of open
subsets has a cofinal monotone sequence), every measure $\mu$ on $X$
with its Borel $\sigma$-algebra restricts to a continuous valuation on
the open sets.  The following theorem shows that, conversely, every
continuous valuation $\nu$ on an LCS-complete space extends to a
measure $\mu$.  We recall that this holds for locally finite
continuous valuations on locally compact sober spaces
\cite{alvarez-manilla00,KL:measureext}.
\begin{lemma}
  \label{lemma:ext:unique}
  Let $\nu$ be a bounded valuation on a topological space $X$.  If
  $\nu$ has an extension to a measure $\mu$ on $\Borel X$, then $\mu$
  coincides with the \emph{crescent outer measure} $\nu^*$ on
  $\Borel X$:
  $\nu^* (E) \eqdef \inf_{\mathcal F} \sum_{C \in \mathcal F} \nu
  (C)$, where $\mathcal F$ ranges over the countable families of
  crescents whose union contains $E$.
\end{lemma}
Note that $\nu (C)$ makes sense by the Smiley-Horn-Tarski theorem.

\begin{proof}
  For every open set $U$, taking $\mathcal F \eqdef \{U\}$, we obtain
  $\nu^* (U) \leq \nu (U) = \mu (U)$.  Conversely, for every countable
  family $\mathcal F$ of crescents $C$ whose union contains $U$,
  $\sum_{C \in \mathcal F} \nu (C) = \sum_{C \in \mathcal F} \mu (C)
  \geq \mu (\bigcup_{C \in \mathcal F} C) \geq \mu (U) = \nu (U)$, so
  $\nu^* (U) = \mu (U)$.

  It is standard that $\nu^*$ defines a measure on the
  $\sigma$-algebra of \emph{measurable sets}, where a subset $A$ of
  $X$ is called measurable if and only if for all subsets $B$ of $X$,
  $\nu^* (B) = \nu^* (B \cap A) + \nu^* (B \diff A)$ (see, e.g.,
  \cite[Theorem~3.2]{KL:measureext}).  We claim that every open set
  $U$ is measurable.  Let us fix a subset $B$ of $X$.  For every
  crescent $C$, $C \cap U$ and $C \diff U$ are crescents again.
  Hence, for every countable family
  $\mathcal F \eqdef {(C_n)}_{n \in \nat}$ of crescents whose union
  contains $B$,
  $\sum_{C \in \mathcal F} \nu (C) = \sum_{n \in \nat} \nu (C_n \cap
  U) + \nu (C_n \diff U) \geq \nu^* (B \cap U) + \nu^* (B \diff U)$.
  Taking infima over $\mathcal F$,
  $\nu^* (B) \geq \nu^* (B \cap U) + \nu^* (B \diff U)$.  Conversely,
  for every countable family $\mathcal F$ of crescents whose union
  contains $B \cap U$, for every countable $\mathcal F'$ of crescents
  whose union contains $B \diff U$, $\mathcal F \cup \mathcal F'$ is a
  countable family of crescents whose union contains $B$, so
  $\nu^* (B \cap U) + \nu^* (B \diff U) \geq \nu^* (B)$, whence the
  equality.  Since the measurable sets contain all the open sets, they
  also contain $\Borel X$.

  Hence we have two measures on $\Borel X$, $\mu$ and $\nu^*$, which
  coincide on the open sets.  In particular, $\mu (X) = \nu^* (X) <
  \infty$, so they are bounded.  It follows that $\mu$ and $\nu^*$
  agree on the whole of $\Borel X$.
\end{proof}

\begin{recapthm}{thm:ext}
  Let $X$ be an LCS-complete space.  \propextaux
\end{recapthm}
\begin{proof}
  Let $\nu$ be a continuous valuation on $X$, and let $X$ be written
  as $\fcap_{n \in \nat} W_n$, where each $W_n$ is open in some
  locally compact sober space $Y$.
  
  Let ${(U_i)}_{i \in I}$ be the family of open subsets of $X$ of
  finite $\nu$-measure.  This is a directed family, since $\nu (U_i
  \cup U_j) \leq \nu (U_i) + \nu (U_j)$.  We write $U_\infty$ for
  $\dcup_{i \in I} U_i$.  If $\nu$ were locally finite, then
  $U_\infty$ would be equal to $X$, but we do not assume so much.
  
  For each $i \in I$, $\nu_{|U_i}$ is a bounded continuous valuation.
  Letting $e \colon X \to Y$ be the inclusion map, the image of
  $\nu_{|U_i}$ by $e$ is another bounded continuous valuation, which
  we write as $\nu'_i$: for every open subset $V$ of $Y$,
  $\nu'_i (V) = \nu_{|U_i} (e^{-1} (V)) = \nu (V \cap U_i)$.  Note
  that $i \sqsubseteq j$ implies $\nu'_i \leq \nu'_j$ (namely,
  $\nu'_i (V) \leq \nu'_j (V)$ for every $V$).

  We claim that $i \sqsubseteq j$ implies that for every crescent $C$,
  $\nu'_i (C) \leq \nu'_j (C)$.  In order to show that, let us write
  $C$ as $U \diff V$, where $U$ and $V$ are open in $Y$.  Replacing
  $V$ by $U \cap V$ if needed, we may assume $V \subseteq U$.  For
  every $k \sqsubseteq j$, we have:
  \begin{align}
    \nonumber
    \nu_{|U_j} (C \cap U_k)
    & = \nu_{|U_j} ((U \cap U_k) \diff (V \cap U_k)) \\ \nonumber
    & = \nu_{|U_j} (U \cap U_k) - \nu_{|U_j} (V \cap U_k) 
    & \text{since }\nu_{|U_j}\text{ is additive on }\mathcal A (\Open X)\\ \nonumber
    & = \nu (U \cap U_k) - \nu (V \cap U_k)
    & \text{since }U_k \subseteq U_j \\
    \label{eq:nui}
    & = \nu_{|U_k} (U) - \nu_{|U_k} (V) = \nu_{|U_k} (C).
  \end{align}
  Taking $k \eqdef i$ in (\ref{eq:nui}),
  $\nu_{|U_i} (C) = \nu_{|U_j} (C \cap U_i)$, which is less than or
  equal to $\nu_{|U_j} (C \cap U_j)$ (the difference is
  $\nu_{|U_j} (C \cap U_j \diff U_i) \geq 0$), and the latter is equal
  to $\nu_{|U_j} (C)$ by (\ref{eq:nui}) with $k \eqdef j$.

  We have seen that $\nu'_i$ extends to a measure $\mu_i$ on $Y$.  By
  Lemma~\ref{lemma:ext:unique}, $\mu_i={\nu'_i}^*$.  Using the formula
  for the crescent outer measure, we obtain that if $i \sqsubseteq j$,
  then $\mu_i (E) \leq \mu_j (E)$ for every $E \in \Borel Y$.

  Since $X$ is $G_\delta$ hence Borel in $Y$, $\Borel X$ is included
  in $\Borel Y$.  Hence $\mu_i$ also defines a measure on the smaller
  $\sigma$-algebra $\Borel X$.  We still write it as $\mu_i$, and we
  note that $i \sqsubseteq j$ implies that $\mu_i (E) \leq \mu_j (E)$
  for every $E \in \Borel X$.  Also, $\mu_i$ extends $\nu_{|U_i}$, as
  we now claim.  Let $U$ be any open subset of $X$.  By definition of
  the subspace topology, $U$ is the intersection of some open subset
  $\widehat U$ of $Y$ with $X$.  $U$ is then equal to
  $\fcap_{n \in \nat} \widehat U \cap W_n$.
  % For each $n \in \nat$, $\nu'_i (\widehat U \cap W_n) = \mu_i
  % (\widehat U \cap W_n)$.
  Now
  $\mu_i (U) = \mu_i (\fcap_{n \in \nat} \widehat U \cap W_n) =
  \inf_{n \in \nat} \mu_i (\widehat U \cap W_n)$ (Kolmogorov's
  criterion)
  $= \inf_{n \in \nat} \nu'_i (\widehat U \cap W_n) = \inf_{n \in
    \nat} \nu_{|U_i} (U)$ (since
  $\widehat U \cap W_n \cap U_i = U \cap U_i$) $= \nu_{|U_i} (U)$.

  Any directed supremum of measures is a measure.  Hence consider
  $\mu (E) \eqdef \dsup_{i \in I} \mu_i (E)$.  For every open subset
  $U$ of $X$,
  $\mu (U) = \dsup_{i \in I} \mu_i (U) = \dsup_{i \in I} \nu_{|U_i}
  (U) = \dsup_{i \in I} \nu (U \cap U_i) = \nu (U \cap U_\infty) =
  \nu_{|U_\infty} (U)$, so $\mu$ extends $\nu_{|U_\infty}$.  Let
  $\iota$ be the indiscrete measure on $X \diff U_\infty$, namely
  $\iota (E)$ is equal to $\infty$ if $E$ intersects
  $X \diff U_\infty$, to $0$ if $E \subseteq U_\infty$.  We check that
  the measure $\mu + \iota$ extends $\nu$.  For every open subset $U$
  of $X$, either $U \subseteq U_\infty$ and
  $\nu (U) = \nu_{|U_\infty} (U) = \mu (U) = \mu (U) + \iota (U)$, or
  $U$ intersects $X \diff U_\infty$, say at $x$.  In the latter case,
  $\iota (U) = \infty$ so $\mu (U) + \iota (U) = \infty$, while
  $\nu (U) = \infty$ because, by definition, $x$ has no open
  neighborhood of finite $\nu$-measure.
\end{proof}

\begin{remark}
  \label{rem:pi02}
  More generally, the proof of Theorem~\ref{thm:ext} would work on
  $\bPi^0_2$ subsets of locally compact sober spaces.  (That is a
  strict extension, by Proposition~\ref{prop:ccmq:UCO}.)  In that
  case, we write $X$ as $\bigcap_{n \in \nat} W_n$ where each $W_n$ is
  the union of a closed and an open set.  Replacing $W_n$ by
  $\bigcap_{i=0}^n W_i$, we make sure that the sequence of sets $W_n$
  is descending, and $W_n$ is still in $\mathcal A (\Open Y)$.  The
  rest of the proof is unchanged.
\end{remark}
\section{Conclusion}
\label{sec:conclusion}

We have given two applications of the theory of LCS-complete spaces
(Theorem~\ref{thm:ext}, Corollary~\ref{corl:homeo:prev}).  We should
mention a final application \cite[Theorem~9.4]{JGL:projlim}, which
will be published elsewhere: given a projective system
${(p_{ij} \colon X_j \to X_i))}_{i \sqsubseteq j \in I}$ of
LCS-complete spaces such that $I$ has a countable cofinal subset,
given locally finite continuous valuations $\nu_i$ on $X_i$ that are
compatible in the sense that for all $i \sqsubseteq j$ in $I$, $\nu_i$
is the image valuation of $\nu_j$ by $p_{ij}$, there is a unique
continuous valuation $\nu$ on the projective limit $X$ of the
projective system such that $\nu$ projects back to $\nu_i$ for every
$i \in I$.  This extends a famous theorem of Prohorov's
\cite{Prohorov:projlim}, which appears as the subcase where each $X_i$
is Polish and each $\nu_i$ is a measure.

One question that remains open, though, is: $(i)$ Is the projective limit
$X$ of a projective system of LCS-complete spaces as above again
LCS-complete?

That is only one of many remaining open questions: $(ii)$ Is every
sober compactly Choquet-complete space LCS-complete?  $(iii)$ Is every
sober convergence Choquet-complete space domain-complete?  $(iv)$ Is
every coherent LCS-complete space a $G_\delta$ subset of a stably
(locally) compact space? $(v)$ Is every $\bPi^0_2$ subset of an
domain-complete space again domain-complete?  (A similar result fails
for LCS-complete spaces, by Proposition~\ref{prop:ccmq:UCO}.)
% This
% question is prompted by Remark~\ref{rem:pi02}: in that case the
% LCS-complete spaces would be exactly the $\bPi^0_2$ subsets of locally
% compact sober spaces.
$(vi)$ Is every \emph{countably correlated} space (i.e., every space
homeomorphic to a $\bPi^0_2$ subset of $\pow (I)$ for some, possibly
uncountable set $I$, see \cite{Chen:qPolish}) LCS-complete?  $(vii)$
Is every LCS-complete space countably correlated?  $(viii)$ Are
regular \v{C}ech-complete spaces LCS-complete, where \v{C}ech-complete
is understood as in \cite[Exercise~6.21]{JGL-topology}?  $(ix)$ Are
all regular LCS-complete spaces \v{C}ech-complete?

% generaliser Frolik?
% image of domain-complete by open continuous map into T0 space is qdomain-complete?
% (see Theorem 40, de Brecht's paper on qPolish spaces)
% locale characterization?
% domain-complete qui sont T_D=scattered? (voir Thm 65 de Brecht)
% si on ajoute un # denombrable de mesurable (ou de Sigma^0_2?) comme
% ouverts, ca reste domain-complete?
% retracts?  (corl 26)
% R x (0,infty)  union Q x {0} est Baire, mais pas completement

% metrisable (https://math.stackexchange.com/questions/3003649/example-of-a-baire-metric-space-which-is-not-completely-metrizable)

\bibliographystyle{entcs}
\ifentcs
\bibliography{gdelta}
\else

\fi
\end{document}